\documentclass[a4paper, 11pt]{article}
\usepackage{amsmath,amssymb,esint,amscd,xspace,fancyhdr,color,authblk,srcltx,fontenc,bbm}
\usepackage{hyperref}
\usepackage{cite}

\newtheorem{theorem}{Theorem}[section]

\newtheorem{corollary}[theorem]{Corollary}
\newtheorem{definition}[theorem]{Definition}
\newtheorem{lemma}[theorem]{Lemma}
\newtheorem{proposition}[theorem]{Proposition}
\newtheorem{remark}[theorem]{Remark}

\newenvironment{proof}[1][Proof]{\textbf{#1.} }{\hfill\rule{0.5em}{0.5em}}
{\catcode`\@=11\global\let\AddToReset=\@addtoreset
\AddToReset{equation}{section}

\AddToReset{theorem}{section}

\title{Global regularity estimates for $p(x)$-Laplacian variational inequalities with singular or degenerate matrix-valued weights}

\author{Minh-Phuong Tran\thanks{Applied Analysis Research Group, Faculty of Mathematics and Statistics, Ton Duc Thang University, Ho Chi Minh city, Vietnam; \texttt{tranminhphuong@tdtu.edu.vn}}; Duc-Quang Bui\thanks{Le Mans Universit\'e, Laboratoire Manceau de Math\'ematiques, Avenue Olivier Messiaen, 72085 Le Mans Cedex~9, France; \texttt{Duc\_Quang.Bui@univ-lemans.fr}}; Thanh-Nhan Nguyen\footnote{Corresponding author} \thanks{Group of Analysis and Applied Mathematics, Department of Mathematics, Ho Chi Minh City University of Education, Ho Chi Minh city, Vietnam; \texttt{nhannt@hcmue.edu.vn}}}

\date{\today} 

\begin{document}
 
\maketitle
\begin{abstract}

We establish the global gradient bounds for weak solutions to the elliptic variational inequality with two-sided obstructions, associated with a $p(x)$-Laplacian type operator involving degenerate or singular matrix weights.  Under the optimal regularity assumptions on the matrix-valued weight, suitable geometric flatness of the domain, and the prescribed data, we aim to investigate the effects of the problem structure on the level of integrability properties of solutions. To this end, we develop regularity in two regards: weighted Calder\'on-Zygmund-type and general weighted Orlicz-type estimates. A notable feature of our results is that, through a constructive level-set approach, the estimates can be derived with minimal dependence of the scaling parameter on the structural constants. The regularity results are then sharp in the sense that they enable the construction of a level-set estimate with nearly optimal scaling parameters, within admissible parameter sets.\\


\medskip


\end{abstract}
\maketitle

\section{Introduction}

$ $
\indent
\textit{1.1. Motivation and relevant studies.} Aside from classical regularity issues regarding uniformly elliptic equations, the analysis of problems driven by degenerate matrix weights has recently garnered the attention of many authors due to their interest and applications. As far as we are concerned, there are few known studies in the literature concerning problems of this type, as well as their regularity and existence theory.

Regarding the aspect of regularity theory, the analysis has started in pioneering works for variational problems with uniformly elliptic/parabolic types (with or without degenerate structures). They have been initiated and developed in a series of fruitful results via the contributions of several authors over the last decades (see classical textbooks~\cite{GT1983, DiBen1993, Lieber1996, Krylov2008, LSU1967}, seminal papers~\cite{Meyer1963, Cianchi2011, LU1970, Frehse2008,AM2005, AM2007, Caffa1998, Caffa1989}, and a rich collection of related references). To simplify our discussion, let us first consider the weak solutions to the uniformly elliptic equations of the linear type 
\begin{align}\label{eq:uni_ellip}
\mathrm{div}[\mathbb{A}(x)\nabla u(x)]=\mathrm{div}[\mathbb{A}(x)\mathbf{F}(x)], \quad x \in \Omega,
\end{align}
where $\Omega \subset \mathbb{R}^n$ is an open bounded subset ($n \ge 2$), $\mathbf{F}: \Omega \to \mathbb{R}^n$ is a given measurable vector field, $\mathbb{A}=(a_{ij})_{n \times n}$ is a symmetric, positive definite matrix with measurable coefficients satisfying the uniform ellipticity:
\begin{align}\label{cond:linear_uniform}
\lambda_1 |\xi|^2 \le \langle \mathbb{A}(x)\xi, \xi \rangle \le \lambda_2 |\xi|^2,
\end{align}
for almost every $x \in \Omega$, for all $\xi \in \mathbb{R}^n$, and some positive constants $0<\lambda_1\le \lambda_2<\infty$. The classical example of such a class of equations or systems is given by the Laplace operator, where $\mathbb{A}=\mathrm{Id}_n$-the identity matrix, and the first set of existence and regularity results goes back to the seminal works by Calder\'on-Zygmund in~\cite{CZ1952, CZ1956}. Via Calder\'on-Zygmund theory of singular integrals, the regularity estimates for the gradient of weak solutions have been established. Since then, many important progresses have been made for a larger class of operators involving matrices $\mathbb{A}$, and we especially refer to~\cite{Meyer1963, Fazio1996, AQ2002, BW2004} for the studies governed by linear uniformly elliptic operators in various settings and under suitable assumptions on the data of problems. 

Be different from the standard uniform ellipticity condition described in~\eqref{cond:linear_uniform}, when the matrix $\mathbb{A}$ in~\eqref{eq:uni_ellip} is controlled by an additional matrix weight $\mu$ that may be infinite or vanish in the domain, our equation is then no longer uniformly elliptic and the behavior of solution changes dramatically. Then, it allows us to consider a new class of degenerate elliptic equations whose ellipticity property depends on $\mu$. Loosely speaking, in this case, our matrix $\mathbb{A}$ becomes singular (or degenerate) when its determinant goes to zero or its inverse blows up at some points. More precisely, the degenerate condition is often written as
\begin{align}\label{cond:linear_degenerate}
\lambda_1 \mu(x)|\xi|^2 \le \langle \mathbb{A}(x) \xi, \xi \rangle \le \lambda_2 \mu(x) |\xi|^2, \quad x \in \Omega, \ \xi \in \mathbb{R}^n,
\end{align}
where $\mu: \Omega \to [0,\infty)$ is a scalar weight function. With the presence of matrix weight $\mu$, the equation we have encountered will be degenerate (or singular) when $\mu$ (or $\mu^{-1}$) is unbounded. A striking example that one has in mind for this case is $\mathbb{A}(x) = |x|^{\pm \varepsilon}\mathrm{Id}_n$ with $\varepsilon>0$ small enough. In this case, the eigenvalues of $\mathbb{A}$ may approach zero or diverge to infinity, which leads to the loss of ellipticity. In general, the main difficulties arise from the fact of working with weight $\mu$, which makes the ellipticity of equations under consideration appears rather sensitive; it entails challenges and significantly restricts qualitative properties of the solutions. 

A special interest in the study of such linear degenerate elliptic equations is motivated by their applications in various fields, such as anisotropic diffusion in porous media, where the permeability/diffusivity tensor acts as a weight~\cite{PS2009, DH1998}. For typical example, when $\mathbb{A}(x) = |x|^{\pm \varepsilon}\mathrm{Id}_n$ as premised above, one describes the porous medium where the permeability/diffusivity vanishes (the material is nearly impermeable, no flows) or blows up (the material is extremely permeable, very strong flow) near the origin. In addition, equations with degenerate-singular weights also appear in mathematical finance via stochastic volatility models~\cite{Fee2014, Fee2015}, mathematical biology~\cite{EM2013}, conical spaces, Poincar\'e-Einstein manifolds in conformal geometry, the study of fractional elliptic equations~\cite{CG2011, STT2020}, and in many different contexts. 

Due to their own interest, the linear problems with degenerate coefficients and regularity properties for their solutions have been systematically investigated in various contributions. For instance, when the weight function $\mu$ is in the $\mathcal{A}_2$ class of Muckenhoupt weights and data $\mathbf{F}$ is nice enough, the H\"older continuity of weak solutions to~\eqref{eq:uni_ellip} was done by Fabes-Kenig-Serapioni in~\cite{FKS1982}. Later, in~\cite{CMP2018}, when $\mu \in \mathcal{A}_2$ and the assumption of smallness weighted BMO condition on $\mathbb{A}$, the $L^p$ and $W^{1,p}$ estimates have been investigated by Cao-Mengesha-Phan in the weighted settings. These results were later extended to global results in~\cite{Phan2020} and to the case of systems in~\cite{CMP2019}. Then, impressive advances have been shown by several authors, and a large number of connections with the topic of degenerate elliptic operators can be found in~\cite{CW1986, FKS1982, FGW1994, FL1982, KRSS2021}, as well as for further reading. In the spirit of Calder\'on-Zygmund estimates, we especially mention the works of Balci-Diening-Giova-Napoli in a celebrated paper~\cite{BDGN2022} and Balci-Byun-Diening-Lee in~\cite{BBDL2023} for the validity of local and global weighted gradient estimates enjoying minimal assumptions on the boundary $\partial\Omega$. It is worth noticing that these two latest papers deal, in fact, with the small BMO condition on $\mathrm{log}\mathbb{A}$ instead of the weighted BMO smallness condition mentioned in~\cite{CMP2018} (see Definition~\ref{def:log-BMO}). As shown, the modification brings many advantages when studying the sharp gradient regularity in a large variety of function spaces. 

The nonlinear equation with matrix-valued degenerate weight, involving $p$-Laplacian operator ($1<p<\infty$), is viewed as a natural generalization of the linear equation of elliptic type~\eqref{eq:uni_ellip}, and can be considered in the divergence form
\begin{align}\label{eq:degenerate_ellip}
\mathrm{div}[|\mathbb{W}(x)\nabla u|^{p-2}\mathbb{W}^2(x)\nabla u] = \mathrm{div}[|\mathbb{W}(x)\mathbf{F}|^{p-2}\mathbb{W}^2(x)\mathbf{F}], \quad x \in \Omega,
\end{align}
where $\mathbb{W}$ is also a symmetric, positive definite matrix satisfying $\mathbb{W} = \mathbb{A}^{\frac{1}{2}}$ and has a uniformly bounded condition number $\Lambda>0$, namely $|\mathbb{W}(x)||\mathbb{W}^{-1}(x)| \le \Lambda$, for all $x \in \Omega$ (the $|\cdot|$ denotes the spectral matrix norm, see Section~\ref{sec:pre}).  Advancing from the linear to the nonlinear case comes with several additional challenges. It is, therefore, interesting and significant to reach regularity structures for the gradient of solutions. We also emphasize here that for nonlinear elliptic equations driven by a degenerate regime in~\eqref{eq:degenerate_ellip}, the works~\cite{BDGN2022, BBDL2023} aforementioned in fact yield dual results, which allow to treat general, both linear and nonlinear degenerate problems. To be more specific, under an \emph{optimal} assumption that the small BMO seminorm is imposed to $\mathrm{log}\mathbb{A}$ and an appropriate geometric structure on the domain $\Omega$, gradient regularity for weak solutions to~\eqref{eq:degenerate_ellip} in weighted Lebesgue spaces has been first established, stating the implication 
\begin{align*}
\mathbf{F} \in L^\gamma(\Omega,\omega^\gamma dx) \Longrightarrow \nabla u \in L^\gamma(\Omega,\omega^\gamma dx), \quad \text{for every} \ \gamma>1,
\end{align*}
where, one writes  $\mu=\Lambda^{-1}\omega^2$ and Muckenhoupt weight $\omega$ acts in a multiplicative form, see Definition~\ref{def:LW_omega} for rigorous details.  Observing the main role of the \emph{$\log$-BMO condition} on the matrix-valued weights in nonlinear degenerate problems, the above-quoted papers have been a source of inspiration for our recent works~\cite{TN2025_NA, NTMN2025} in the line of research, where the generalized regularity estimates for solutions have been offered for both equations and obstacle problems in the context of variational inequalities. 

Starting from linear equations, the regularity theory has now expanded to a larger class of nonlinear elliptic equations, especially when operators have polynomial growth, involving $p(\cdot)$-Laplacian. Problems with variable exponents appear naturally in several physical problems and have been widely investigated by several contributions. The emblematic applications of nonlinear models with variable exponents come from the geometry of composite materials; the modeling of electrorheological and non-Newtonian fluids; gas or fluid flows through porous media; elasticity; image processing, and many others. We refer the readers to~\cite{Ruzicka2000, AM2005, RR2001, AS2005, CLR2006} and the references therein for more background on applications. It is standard to see that for these sets of problems, the exponents can vary according to the point; the classical frameworks are therefore not enough to derive the gradient regularity results. This limit is the primary motivation in our study to complement and extend the global optimal regularity to the case of more general problems with variable growth exponents, whose nonlinearities involve degenerate matrix weights $\mathbb{W}$, as mentioned earlier.

Further, the study of regularity theory for variational inequalities and free boundary problems, especially the obstacle problems, with numerous applications (such as elastic-plastic torsion, Hele-Shaw flow, tug-of-war game, pricing of American-type options in finance, etc), have also attracted enormous attention from the community during the past decades, see for example~\cite{Figalli2018, FRS2020, FGKS2024}. Roughly speaking, in the case of  two-sided obstacle problem, for given functions $\phi_1 \le \phi_2$ smooth enough, ones wish to construct the solutions $u$ that lie between the obstacles $\phi_1, \phi_2$ and conclude regularity properties of such solutions. It came to our attention that regularity results for obstacle problems can be seen as the extension of theoretical results developed for obstacle-free problems. However, a new class of elliptic operators, when one considers obstacle problems involving both variable exponent growth and singular-degenerate matrix weights, the key point of difficulty is concerned with the nonlinearity of the context, and it is highly non-trivial to deal with. Therefore, we believe that it would be mathematically interesting and significant to reach for regularity estimates for the gradient of solutions to these problems. And as far as we are concerned, the regularity results for such a class of nonlinear problems have not yet been completely investigated in the literature. This motivates us to further treat the regularity estimates, towards a better understanding of the solutions' properties. As such, our main results in this paper, with the self-contained techniques, are new in some sense. 

\textit{1.2. Problem setting.} In the spirit of the calculus of variations, the obstacle problems are closely connected with variational inequalities and free boundary problems, for which the solution is considered as a minimizer of a constrained energy functional over an admissible set of functions. Motivated by numerous models and recent studies on the smoothness properties of minimizers of integral functionals, the problem of interest in this paper is related to
\begin{align}\label{eq:general}
\text{minimize} \ \int_\Omega{\mathcal{P}(x,v(x),\nabla v(x))}, \quad  \text{among all}\ \  \phi_2 \le v \le \phi_2,
\end{align}
in some open bounded subset $\Omega \subset \mathbb{R}^n$ with nonsmooth boundary $\partial\Omega$, for $n \ge 2$; $\phi_1,\phi_2$ are two given functions; and $\mathcal{P}:\Omega \times \mathbb{R} \times \mathbb{R}^n \to \mathbb{R}$ is a Carath\'eodory-regular integrand satisfying a nonstandard growth condition. The study of regularity issues of minimizers of functional integrals, encompassing~\eqref{eq:general} or weak solutions to the associated variational inequality, is the subject of a large body of literature (see, for example, the references in~\cite{AM2005, BLOP2019, EH2010}). Moreover, in the context of problems involving matrix-valued weights, we shall include a matrix-valued function $\mathbb{W}: \Omega \to \mathbb{R}^{n\times n}_{\mathrm{sym}^+}$ that has a uniformly bounded condition number $\Lambda\ge 1$, which is equivalent to: 
\begin{align}\label{Pi-2}
\Lambda^{-1}\omega(x) \mathrm{Id}_n \le \mathbb{W}(x) \le \omega(x)\mathrm{Id}_n, \quad \text{or} \quad \Lambda^{-1}\omega(x)|\xi| \le |\mathbb{W}(x)\xi| \le \omega(x)|\xi|,
\end{align}
for almost every $x \in \Omega$ and all $\xi \in \mathbb{R}^n$, if one defines a new scalar weight $\omega:=|\mathbb{W}|$. Here, we propose the investigation of variable exponent, to set up the growth condition, let us consider the continuous function $p: \Omega \to (1,\infty)$ satisfying
\begin{align}\label{cond:pq1}
1 < p_{\mathrm{min}} \le p(\cdot) \le p_{\mathrm{max}} < \infty.
\end{align}
Due to the presence of a $p$-Laplacian-type operator with a variable exponent, it should be chased by considering special classes of function spaces. Suppose that $\omega^{p(\cdot)}$ belongs to the class of $\mathcal{A}_{p(\cdot)}$-Muckenhoupt weights (see Definition~\ref{def:Muck_VE}), the standard analysis on Lebesgue and Sobolev function spaces is then understood in the multiplicative weighted frameworks:  $L^{p(\cdot)}\big(\Omega,\omega^{p(\cdot)}\big)$ and $W^{1,p(\cdot)}\big(\Omega,\omega^{p(\cdot)}\big)$, respectively. And we send the reader to Section~\ref{sec:pre} for more details concerning their definitions. 

Given datum vector field $\mathbf{F}: \Omega \to \mathbb{R}^n$ such that $|\mathbf{F}| \in L^{p(\cdot)}\big(\Omega,\omega^{p(\cdot)}dx\big)$ and $\phi_1, \phi_2 \in W^{1,p(\cdot)}\big(\Omega,\omega^{p(\cdot)}dx\big)$ are two fixed obstacle functions such that 
\begin{align}\label{def-phi-12}
\phi_1\le \phi_2 \mbox{ almost everywhere in } \Omega \mbox{ and } \phi_1 \le 0 \le \phi_2 \mbox{ on } \partial\Omega,
\end{align}
we introduce the following convex set
\begin{align}\label{def-K}
\mathbb{K}_{\phi_1,\phi_2}^{\omega} := \left\{\phi \in W^{1,p(\cdot)}_{0}\big(\Omega,\omega^{p(\cdot)}\big): \ \phi_1 \le \phi \le \phi_2 \ \mbox{ a.e. in } \Omega\right\}.
\end{align}
In the present paper, our main two-obstacle problem can be nicely characterized by a variational inequality regarding~\eqref{eq:degenerate_ellip} in the variable exponent setting. More precisely, we mainly refer to the degenerate elliptic variational inequality of the type
\begin{align}\label{OP-var}
& \int_{\Omega} |\mathbb{W}(x)\nabla u|^{p(x)-2}\mathbb{W}^2(x)\nabla u \cdot  \nabla(u-\phi) dx  \le   \int_{\Omega} |\mathbb{W}(x)\mathbf{F}|^{p(x)-2}\mathbb{W}^2(x)\mathbf{F} \cdot \nabla (u-\phi) dx,
\end{align}
that holds for all test functions $\phi \in \mathbb{K}_{\phi_1,\phi_2}^\omega$ such that $\phi - u \in W^{1,p(\cdot)}_{0}\big(\Omega,\omega^{p(\cdot)}\big)$. We say that a weak solution to a two-obstacle problem, referring to degenerate-singular $p(\cdot)$-Laplacian type, is a map $u \in \mathbb{K}_{\phi_1,\phi_2}^{\omega}$ satisfying the variational inequality~\eqref{OP-var}. Although the emphasis of the paper is not on the existence, as $\omega^{p(\cdot)} \in \mathcal{A}_{p(\cdot)}$-Muckenhoupt class, the existence of weak solutions to~\eqref{OP-var} is pointed out based on the maximal monotone operator theory and variational arguments from the calculus of variations. 

Without any obstacles, the mathematical description turns back to the regularity theory for nonlinear equations with degenerate weights, and one can go through the works~\cite{CMP2018, Phan2020, CMP2019, BBDL2023, BDGN2022} for the precise statements of Calder\'on-Zygmund estimates discussed above. On the other hand, with \emph{constant exponent} $p$,  it is remarkable that the problem~\eqref{OP-var} reduces to double-obstacle problems driven by $p$-Laplacian with degenerate weights, we refer the reader to the work~\cite{BBDL2023} for Calder\'on-Zygmund estimates for weak solutions (with one-sided obstacle); and the generalized gradient bound has been addressed in our previous work~\cite{TN2025_NA} (with two-sided obstacle) under separate assumptions on the nonlinearities and the boundary of domains. In the last case, without the presence of degenerate weights, our obstacle problem is closely related to $p(\cdot)$-Laplacian, which has been studied by a myriad of investigations. 

\textit{1.3. Assumptions.} Before presenting and discussing the main achievement of this paper, let us specify the assumptions on the original problem under which we consider in this paper.\\[5pt] 
\textbf{Assumption $\mathbf{(H_1)}$:} Regarding the variable exponent, for which it allows the function $p$ to depend on the spatial variable $x$ satisfying~\eqref{cond:pq1} and the $\log$-H\"older continuity in $\Omega$ for some $\kappa \in \left( 0,\frac{1}{8} \right)$.  This assumption means that there exist constants $\kappa \in \left( 0,\frac{1}{8} \right)$, $r_0>0$ and a non-decreasing continuous function $\alpha: \mathbb{R}^+ \to \mathbb{R}^+$, $\alpha(0) = 0$ such that
\begin{align}\label{cond:ap}
\begin{cases}
|p(x) - p(y)| \le \alpha(|x-y|), \ \mbox{ for all } x, y \in \Omega,\\
\displaystyle{\sup_{0<r\le r_0}\alpha(r)\mathrm{log}\left(\frac{1}{r}\right)} \le \kappa,\ \ \text{or equivalently,} \quad \displaystyle{\sup_{0<r\le r_0}{\left( \frac{1}{r}\right)^{\alpha(r)}}} \le e^{\kappa}.
\end{cases}
\end{align} 
Let us briefly discuss the assumption $\mathbf{(H_1)}$ on variable exponent $p: \Omega \to (1,\infty)$. To our knowledge,~\eqref{cond:pq1} and $\log$-H\"older condition~\eqref{cond:ap} are important regularity assumptions to guarantee the oscillation of the function $p(\cdot)$ is well controlled and therefore, the boundedness of the maximal operators (in the corresponding variable $L^{p(\cdot)}$-spaces) and other results in harmonic analysis still hold in the setting of variable exponent (see~\cite{DHHR2017}). \\[5pt]
\textbf{Assumption $\mathbf{(H_2)}$:} The elliptic degenerate matrix-valued weight, as mentioned above, is defined as a map $\mathbb{W}: \Omega \to \mathbb{R}^{n\times n}_{\mathrm{sym}^+}$, which has a uniformly bounded condition number $\Lambda \ge 1$ (equivalently described in~\eqref{Pi-2}) and satisfies the so-called small $\log$-BMO seminorm. This condition allows the existence of a small constant $\kappa>0$ such that 
\begin{align}
\label{eq:logBMO_H2}
|\mathrm{log}\mathbb{W}|_{\mathrm{BMO}} \le \kappa.
\end{align}
It remarks here that in our strategy of proof, the value of such a small constant may change from line to line, depending on the arguments conducted in the context. Here, we also refer to Definition~\ref{def:log-BMO} and Remark~\ref{rem:H2} for a rigorous definition and interpretation of $|\mathrm{log}\mathbb{W}|_{\mathrm{BMO}}$, the $\log$-BMO seminorm of matrix weight $\mathbb{W}$.\\[5pt]
\textbf{Assumption $\mathbf{(H_3)}$:} To study the global gradient regularity, it is natural to make some regularity assumptions on the ground domain $\Omega$ because the geometric property of $\Omega$ strongly affects the regularity near the boundary points.  Inspired by the works of Byun \emph{et al.} in~\cite{BW2004, BY24, BLOP2019}, in our investigation, $\Omega$ is assumed to satisfy the $(\kappa,r_0)$-Reifenberg flat condition for some small constants $\kappa>0$ and $r_0>0$. The reader may refer to Section~\ref{sec:pre} for detailed definition and discussion on this minimal regularity assumption. 

It is worth remarking here that the assumptions $(\mathbf{H}_1)$, $(\mathbf{H}_2)$, and $(\mathbf{H}_3)$ are always linked to the couple $(\kappa,r_0)$. Further, the scaling parameter $\kappa>0$ is at our disposal, and its value may vary along the lines of proofs. The same letter $\kappa$ will generally be used to denote different parameters in the course of argument, and finally, $\kappa>0$ will possibly be chosen as the smallest one in an admissible set of such values. 

\textit{1.4. Main results and contribution of this work.} Our aim is, in particular, to highlight a global gradient bound for weak solutions to~\eqref{OP-var} in the setting of general Orlicz spaces under minimal assumptions on data, nonsmooth domain, and the nonlinearity of the differential operator involving degenerate matrix-valued weights. More precisely, $u \in \mathbb{K}_{\phi_1,\phi_2}^{\omega}$ a weak solution of~\eqref{OP-var} under assumptions $\mathbf{(H_1)}$-$\mathbf{(H_3)}$, we establish the following optimal global Cader\'on-Zygmund-type estimates
\begin{align}\label{ineq-main-1}
\int_{\Omega} \omega^{\gamma p(x)}|\nabla u|^{\gamma p(x)} dx \le C \left(1 + \int_{\Omega} [\mathbb{F}_{\omega}(x)]^{\gamma} dx\right),
\end{align}
for all $\gamma>1$, for some positive constant $C$ independent of data $\mathbf{F}$, obstacle functions $\phi_1,\phi_2$ and $u$. Here, for notational simplicity, we employ 
\begin{align}\label{def-UF}
\mathbb{F}_{\omega}(x) = \omega^{p(x)} \left(|\mathbf{F}|^{p(x)} + |\nabla \phi_1|^{p(x)} + |\nabla \phi_2|^{p(x)}\right), \quad x \in \Omega.
\end{align}
A distinctive trait of our results is that the estimates on level-sets will possibly be tracked with the least dependence of the chosen scaling parameter on the structural constants of the proposed method. The regularity result is then optimal in the sense that it facilitates the construction of a level-set estimate on feasible scaling parameter sets. We refer the reader in particular to the precise statement and argument of Theorem~\ref{theo-LV}, concerning the optimal dependence of scaling coefficients.

We also naturally derive a more general gradient estimate for solutions in the weighted Orlicz spaces, controlled by the \emph{fractional maximal operator}, that has profound and intriguing connections with fractional derivatives/integrals in the sense of Calder\'on spaces (see~\cite{KM2012}). In particular, we show a global estimate for weak solutions to~\eqref{OP-var} under the structure conditions $\mathbf{(H_1)}$-$\mathbf{(H_3)}$ of the following type
\begin{align}\label{ineq-main-2}
\int_{\Omega}\Phi\left(\mathbf{M}_{\beta}\big(\omega^{p(\cdot)} |\nabla u|^{p(\cdot)}\big)\right)dx & \le C \left[ 1 +  \int_{\Omega}\Phi\left(\mathbf{M}_{\beta}\mathbb{F}_{\omega}\right)dx\right],
\end{align}
for any increasing continuously differentiable function $\Phi: [0,\infty) \to [0,\infty)$ with $\Phi(0)=0$ and satisfying a doubling condition, see Definition~\ref{def:orlicz}. Here, it should mention $\mathbf{M}_\beta$-the fractional maximal operator, for $\beta \in [0,n)$, that will be crucial in our proofs later, see Definition~\ref{def:M_beta} and further discussions in Section~\ref{sec:pre}. Theorem~\ref{theo-Ge2} provides a detailed statement of this new result in weighted Orlicz spaces, which will be addressed in Section~\ref{sec:main}. 

\textit{1.5. Structure of the paper.} After the introductory section with an outlook on the considered problem, the existing results, and our main objectives, this paper is now organized as follows. In the next section, we first set generic notation, briefly recapitulate some standard definitions and technical preliminary lemmas that will be necessary for the rest of this paper. In Section~\ref{sec:comparison}, we shall deal with comparison principles, which form the key tool in our argument, and cause a number of new challenges to deal with due to the two-sided obstacles and the structure of the operator. Later, Section~\ref{sec:main} will describe in detail the new results obtained in this study via the formal statements of Theorem~\ref{theo-Ge1} and~\ref{theo-Ge2}, respectively. And the proofs of the main results are also exposed in this section. Finally, Section~\ref{sec:future} is devoted to some conclusions and discussions in which the novelty is incorporated through degenerate-singular models and our method. Further, some possible topics discussed in this section might also be of interest to many authors in this field. 

\section{Preliminaries}\label{sec:pre}

This section is devoted to some notation, basic definitions, and necessary background on function spaces that our study relies on. Furthermore, this section also contains preliminary lemmas regarding the two-obstacle problem that are useful throughout the text.

\textit{2.1. Basic notation and definitions.} We first record a minimum amount of notations needed for our analysis later. Throughout the paper, let $n \ge 2$ and $\Omega \subset \mathbb{R}^n$ be an open bounded subset with non-smooth boundary $\partial\Omega$, with its diameter often denoted by $\mathrm{diam}(\Omega)$. For any point $x_0 \in \mathbb{R}^n$ and $R>0$, we define $B_R(x_0)$ the Euclidean open ball centered at $x_0$ with radius $R$. When the information of the center is clear from the context, the simplified notations $B_R$ and $\Omega_R:=\Omega \cap B_{R}$ will be employed if there is no confusion caused. Moreover, if $\mathcal{V} \subset \mathbb{R}^n$ is a measurable set, we shall denote $|\mathcal{V}|$ its Lebesgue measure. In this study, by the writing $\mathcal{M}eas(\Omega)$, it implicitly means the set of all Lebesgue measurable functions on $\Omega$. And for $f \in \mathcal{M}eas(\Omega)$, to simplify the notation in later arguments, we simply write $\{x \in \Omega: \, |f(x)|>t\}$ for some $t \ge 0$ by $\{|f|>t\}$. For an integrable map $f: \mathcal{V} \subset \Omega \to \mathbb{R}$ and $\mathcal{V}$ has finite positive measure, $\fint_{\mathcal{V}}$ or $\overline{f}_{\mathcal{V}}$ denotes the average value of $f$, defined as
$$\overline{f}_{\mathcal{V}} := \fint_{\mathcal{V}} f(x) dx = \frac{1}{|\mathcal{V}|} \int_{\mathcal{V}} f(x) dx.$$
Here, we say that $\omega$ is a (scalar) weight function if $\omega: \Omega \to \mathbb{R}^+$ is measurable and positive almost everywhere in $\Omega$. It is to be noticed that in the general context of the problem~\eqref{OP-var}, for a weight $\omega \in L^1_{\mathrm{loc}}(\mathbb{R}^n;\mathbb{R}^+)$ and a measurable subset $\mathcal{V} \subset \mathbb{R}^n$, it often writes
$$\omega(\mathcal{V}) :=\int_{\mathcal{V}}\omega(x)dx.$$ 
In our estimates, all generic positive constants whose exact value is not important for the purpose will be denoted by $C$ ($C$ may differ in a single chain of inequalities). We often represent positive constants $C(\cdot)$ with parentheses, including the parameters on which the constant depends. Moreover, \emph{data of problem}, which stays in our initial settings of the problem, will be collected in a set of relevant prescribed parameters. And for the sake of brevity, we introduce
\begin{align*}
\textsc{dataset} := \{\beta,n,p_{\mathrm{min}},p_{\mathrm{max}},r_0/\mathrm{diam}(\Omega),\Lambda\}.
\end{align*}

We denote by $\mathbb{R}^{n\times n}_{\mathrm{sym}}$ the set of all symmetric, real-valued square matrices of order $n$. To make the notation fairly self-explanatory, we further identify its subset $\mathbb{R}^{n\times n}_{\mathrm{sym}^+}$, including all symmetric and positive definite matrices. In this space, we consider the spectral norm, defined as follows
$$|\mathsf{M}| = \sup_{|\zeta|\le 1} |\mathsf{M}\zeta|, \ \mbox{ for } \ \mathsf{M} \in \mathbb{R}^{n \times n}_{\mathrm{sym}}.$$ 

\begin{definition}[Matrix weight]
\label{def:matrix_weight}
A matrix-map $\mathbb{W}: \Omega \to \mathbb{R}^{n \times n}_{\mathrm{sym}^+}$ is called a matrix-valued weight if $\mathbb{W}$ is positive definite for almost every $x \in \Omega$. Equivalently, $\mathbb{W}(x)$ is a symmetric, positive definite matrix for almost every $x \in \Omega$. 
\end{definition}
\begin{definition}[Logarithmic mean of weight]
\label{def:log_mean}
The logarithmic mean of a weight $\omega$ over a ball $B \subset \mathbb{R}^n$ will be defined by
\begin{align}\label{log-w}
\langle \omega\rangle_B^{\mathrm{log}} := \exp \left(\fint_B \mathrm{log}(\omega(x)) dx\right).
\end{align}
\end{definition}

\begin{remark} $ $
\begin{itemize}
\item[-] It is remarkable that the map $\exp: \mathbb{R}^{n \times n}_{\mathrm{sym}} \to \mathbb{X}_+$ and its inverse $\mathrm{log}: \mathbb{R}^{n \times n}_{\mathrm{sym}^+} \to \mathbb{R}^{n \times n}_{\mathrm{sym}}$ are defined by transforming the matrix into diagonal form via the spectral decomposition. More precisely, for every matrix $\mathsf{M} \in \mathbb{R}^{n \times n}_{\mathrm{sym}^+}$, it allows us to define the exponential matrix $\exp(\mathsf{M})$ and its inverse $\mathrm{log}(\mathsf{M})$ by convergent Taylor series. 
\item[-] From Definition~\ref{def:log_mean}, it is not difficult to verify that
\begin{align}\label{pro-1-w}
\left\langle \frac{1}{\omega}\right\rangle_B^{\mathrm{log}}\cdot \langle \omega\rangle_B^{\mathrm{log}} = \exp\left(-\fint_B \mathrm{log}(\omega(x))dx\right) \exp\left(\fint_B \mathrm{log}(\omega(x))dx\right) = 1. 
\end{align}
Moreover, another interesting property is that, for $q \in (1,\infty)$, $\omega^q \in \mathcal{A}_{q}$ (Muckenhoupt class) if and only if the following two inequalities hold: 
\begin{align}\notag
\left(\fint_B \omega(x)^q dx \right)^{\frac{1}{q}} \le C \langle \omega\rangle_B^{\mathrm{log}}, \ \mbox{ and } \ \left(\fint_B \omega(x)^{-\frac{q}{q-1}}dx\right)^{1-\frac{1}{q}} \le C \langle 1/\omega\rangle_B^{\mathrm{log}}.
\end{align}
for any ball $B \subset \mathbb{R}^n$. 
\item[-] If $\mathbb{W}$ is a matrix weight as in Definition~\ref{def:matrix_weight} and $B \subset \mathbb{R}^n$ is an open ball, then \begin{align}\label{log-PI}
\langle \mathbb{W}\rangle_B^{\mathrm{log}} := \exp \left(\fint_B \mathrm{log} \mathbb{W}(x) dx\right),
\end{align}
and in the same spirit of~\eqref{pro-1-w}, from~\eqref{log-PI}, we also are able to validate
\begin{align*}
\langle \mathbb{W}^{-1}\rangle_B^{\mathrm{log}} \langle \mathbb{W}\rangle_B^{\mathrm{log}} = \left[\exp\left(-\fint_B \mathrm{log}\mathbb{W}(x)dx\right)\right]  \left[\exp \left(\fint_B \mathrm{log}\mathbb{W}(x)dx\right)\right]^{-1} = \mathrm{Id}_n.
\end{align*}
\end{itemize}
\end{remark}

\begin{definition}[$\log$-BMO seminorm]
\label{def:log-BMO}
Let $B$ be an open ball in $\mathbb{R}^n$ and $\mathbb{W}$ be a matrix weight. The $\log$-$\mathrm{BMO}$ semi-norm of $\mathbb{W}$ over $B$ is defined by
\begin{align}\label{BMO-norm}
|\mathrm{log} \mathbb{W}|_{\mathrm{BMO}(B)} := \sup_{B_\rho \subset B} \fint_{B_\rho} |\mathrm{log} \mathbb{W} - \langle \mathbb{W}\rangle_{B_\rho}^{\mathrm{log}}| dx,
\end{align}
where the supremum is taken over all balls $B_\rho \subset B$. For simplicity, we simply write $|\mathrm{log} \mathbb{W}|_{\mathrm{BMO}}$ whenever $B \supset \Omega$.  
\end{definition}

\begin{remark}
\label{rem:H2}
Let us here briefly discuss the assumption $\mathbf{(H_2)}$, regarding the small $\log$-BMO condition described in~\eqref{eq:logBMO_H2} for the matrix weight $\mathbb{W}$. A natural question arises here: what is the role of the logarithm in this case? And why is the term $\mathrm{log}\mathbb{W}$ bounded instead of $\mathbb{W}$ in proving regularity? - Let us incidentally go back to the structure of problem~\eqref{OP-var}, the uniform ellipticity is violated, and it requires an extra condition on $\mathbb{W}$ (whose values may change in a large range) to control its oscillation. Fabes-Kenig-Serapioni in~\cite{FKS1982} first imposed the scalar weight $\omega$ in~\eqref{Pi-2} to belong to $\mathcal{A}_2$-class to control the integrability of weight for linear problems, Cao-Mengesha-Phan~\cite{CMP2018} added a smallness condition on the weighted BMO norm of $\mathbb{W}$. Later, a notable breakthrough with the smallness condition on the BMO norm for $\mathrm{log}\mathbb{W}$ (or for $\mathrm{log}\mathbb{A}=\frac{1}{2}\mathrm{log}\mathbb{W}$) was proposed by Balci-Diening-Giova-DiNapoli~\cite{BDGN2022} and Balci-Byun-Diening-Lee~\cite{BBDL2023}. The \emph{small $\log$-BMO condition} is stated as a minimal regularity assumption to control the oscillation of the matrix-weight's eigenvalues. It is known that with the standard logarithm, it transforms a multiplicative structure into an additive one, and therefore, the size and scaling of the matrix (determined by its eigenvalues) are well controlled. This makes it very interesting to deal with the degree of oscillation, integrability, and regularity. Assembling these reasons, and thanks to the inspiration from the papers mentioned above, we also make use of the assumption $\mathbf{(H_2)}$, which is minimal, in a sense, for gradient regularity to hold. 
\end{remark}

For a continuous function $p: \Omega \to (1,\infty)$ satisfying~\eqref{cond:pq1} and $\log$-H\"older condition~\eqref{cond:ap}, let us define the modular function as below:
\begin{align}\notag
\mathrm{mod}^{p(\cdot)}(f) = \int_{\Omega} |f(x)|^{p(x)} dx, \quad \mbox{ for } f \in \mathcal{M}eas(\Omega).
\end{align}  

\begin{definition}[Variable Lebesgue spaces]
\label{def:LW}
Given variable function $p(\cdot)$ satisfying conditions in~\eqref{cond:pq1} and~\eqref{cond:ap}, then the variable Lebesgue space $L^{p(\cdot)}(\Omega)$ is defined by the set of all functions $f \in \mathcal{M}eas(\Omega)$ for which $\mathrm{mod}^{p(\cdot)}(f)$ is finite. Moreover, $L^{p(\cdot)}(\Omega)$ is endowed with the Luxemburg norm given by
\begin{align*}
\|f\|_{L^{p(\cdot)}(\Omega)} := \inf \left\{s>0: \ \mathrm{mod}^{p(\cdot)}\left(\frac{f}{s}\right) \le 1 \right\}.
\end{align*}
\end{definition}

\begin{definition}[Variable Lebesgue and Sobolev spaces in the intrinsic weighted setting]
\label{def:LW_omega}
For given continuous function $p(\cdot)$ satisfying conditions~\eqref{cond:pq1},~\eqref{cond:ap} and a weight function $\omega: \Omega \to \mathbb{R}^+$, assume that $\omega^{p(\cdot)} \in L^1_{\mathrm{loc}}(\Omega)$ and $\omega^{-p'(\cdot)} \in L^1_{\mathrm{loc}}(\Omega)$. Then, the variable Lebesgue space in the weighted setting, denoted by $L^{p(\cdot)}(\Omega,\omega^{p(\cdot)})$, is defined as the set of all measurable functions $f$ such that $\omega f \in L^{p(\cdot)}(\Omega)$ (as presented in Definition~\ref{def:LW}). In addition, this space is also a Banach space equipped with the following norm 
\begin{align*}
\|f\|_{L^{p(\cdot)}(\Omega,\omega^{p(\cdot)})} := \|\omega f\|_{L^{p(\cdot)}(\Omega)}.
\end{align*}
In this case, we further define the weighted variable Sobolev space with multiplicative weight $W^{1,p(\cdot)}(\Omega,\omega^{q(\cdot)})$, as follows
\begin{align*}
W^{1,p(\cdot)}(\Omega,\omega^{p(\cdot)}) = \left\{f \in L^{p(\cdot)}(\Omega,\omega^{p(\cdot)}): \, |\nabla f| \in L^{p(\cdot)}(\Omega,\omega^{p(\cdot)}) \right\},
\end{align*}
and this space is also endowed with the norm 
$$\|f\|_{W^{1,p(\cdot)}(\Omega,\omega^{p(\cdot)})} = \|f\|_{L^{p(\cdot)}(\Omega,\omega^{p(\cdot)})} + \|\nabla f\|_{L^{p(\cdot)}(\Omega,\omega^{p(\cdot)})}.$$ 
We also write $W^{1,p(\cdot)}_{0}(\Omega,\omega^{p(\cdot)})$ as a subspace of $W^{1,p(\cdot)}(\Omega,\omega^{p(\cdot)})$ with zero-values on the boundary.  
\end{definition}

\begin{definition}[Variable Lebesgue and Sobolev spaces in the standard weighted setting]
Let $p(\cdot)$ be a continuous function satisfying conditions~\eqref{cond:pq1},~\eqref{cond:ap} and a weight function $\omega: \Omega \to \mathbb{R}^+$. For a fixed constant $q \in [p_{\mathrm{min}},p_{\mathrm{max}}]$, we define the weighted Lebesgue space, denoted by $L^{q}(\Omega,\omega^{q})$, is the set of all functions $f \in \mathcal{M}eas(\Omega)$ such that $\omega f \in L^q(\Omega)$. Moreover, in the (standard) weighted setting, we also consider the variable Lebesgue space $L^{p(\cdot)}(\Omega,\omega^{q})$, that is tailored for our purposes, with the following modular function
\begin{align}\notag
\mathrm{mod}^{p(\cdot)}_{\omega,q}(f) = \int_{\Omega} \omega^{q} |f(x)|^{p(x)}  dx, \quad \mbox{ for } f \in \mathcal{M}eas(\Omega).
\end{align}
Similar to the classical Sobolev spaces, the Sobolev spaces in this weighted setting are built upon the associated Lebesgue spaces. We will use the notations $W^{1,q}(\Omega,\omega^{q})$ and $W^{1,p(\cdot)}(\Omega,\omega^{q})$ with respect to Lebesgue spaces $L^{q}(\Omega,\omega^{q})$ and $L^{p(\cdot)}(\Omega,\omega^{q})$.
\end{definition}

\begin{definition}[Muckenhoupt class]
\label{def:Muck_classical}
Let $1 < s < \infty$, the Muckenhoupt class $\mathcal{A}_s$ is specified by the set of all weights $\omega \in L^1_{\mathrm{loc}}(\mathbb{R}^n)$ satisfying
\begin{align*}
[\omega]_{\mathcal{A}_{s}} := \displaystyle \sup_{B \subset \mathbb{R}^n} \left(\fint_{B} \omega dx\right) \left(\fint_{B} \omega^{-\frac{1}{s-1}} dx\right)^{s-1} < \infty,
\end{align*}
where $[\omega]_{\mathcal{A}_{s}}$ is called the $\mathcal{A}_s$-characteristic of $\omega$. Then, $\omega$ is called an $\mathcal{A}_s$-Muckenhoupt weight. Moreover, the Muckenhoupt class $\mathcal{A}_{\infty}$ is defined by the union of all $\mathcal{A}_{s}$ for all $s > 1$, which means that
$$\mathcal{A}_{\infty} = \displaystyle{\bigcup_{s > 1} \mathcal{A}_{s}}.$$
We recall that if $\mu \in \mathcal{A}_{\infty}$ then there exist $C>0$ and $\theta>0$ such that
\begin{align}\label{Muck-w}
\mu(\mathcal{V}) \le C \left[\frac{|\mathcal{V}|}{|B|}\right]^{\theta} \mu(B), 
\end{align}
for all ball $B \supset \mathcal{V}$ and measurable set $\mathcal{V}$. If \eqref{Muck-w} is valid, we write $[\mu]_{\mathcal{A}_{\infty}} = (C,\theta)$.
\end{definition}

\begin{definition}[Muckenhoupt class in the variable exponent setting]
\label{def:Muck_VE}
Let $p(\cdot)$ be a continuous function satisfying conditions~\eqref{cond:pq1},~\eqref{cond:ap} and a weight function $\omega: \Omega \to \mathbb{R}^+$. Then, we say that the weight $\omega^{p(\cdot)}$ belongs to the $\mathcal{A}_{p(\cdot)}$-Muckenhoupt class if and only if
\begin{align}\notag
\sup_{B} \max_{q \in \{q_1,q_2\}} \frac{1}{|B|^{p_B}} \left(\fint_{B} \omega(x)^{p(x)}dx\right) \left(\fint_{B} \omega(x)^{-p'(x)} dx\right)^{q} < \infty, 
\end{align}
where the supremum is taken over all balls $B$ in $\mathbb{R}^n$, and $q_1$, $q_2$, $p_B$ are respectively given by 
$$q_1 = \min_{x\in B}(p(x)-1), \ q_2 = \max_{x\in B}(p(x)-1) \ \mbox{ and } \ p_B = \left(\fint_B p(x) dx\right)^{-1}.$$
\end{definition}

\begin{definition}[$(\kappa,r_0)$-Reifenberg flatness]
Given $\kappa>0$ and $r_0>0$, we say that the domain $\Omega$ satisfies  $(\kappa,r_0)$-Reifenberg flatness, if for any $x \in \partial \Omega$ and $0<r<(1-\kappa)r_0$, there exists a coordinate system $\{\tilde{x}_1, \tilde{x}_2,\cdots,\tilde{x}_n\}$ such that the new origin $\tilde{\mathbf{0}} \in \Omega$, $x =(0,0,...,0, - r\kappa/(1-\kappa)\tilde{x}_n)$ and
\begin{align}\notag 
B(\tilde{\mathbf{0}},r) \cap \{\tilde{x}_n > 0\} \subset B(\tilde{\mathbf{0}},r) \cap \Omega \subset B(\tilde{\mathbf{0}},r) \cap \left\{\tilde{x}_n >  \frac{-2 r\kappa}{1-\kappa}\right\},
\end{align}
where, we will rather write $\{\tilde{x}_n > \tau\}$ instead of $\{\tilde{x} = (\tilde{x}_1, \tilde{x}_2,\cdots, \tilde{x}_n): \ \tilde{x}_n > \tau\}$, with an intentional abuse of notation.
\end{definition}

\begin{remark}
\label{rem:H3}
The concept of a nonsmooth domain satisfying $(\kappa, r_0)$-Reifenberg flatness was first introduced by Reifenberg in~\cite{Reifenberg} in the minimal surface theory and explored by Byun-Wang in~\cite{BW2004} in the regularity theory. This type of domain, whose boundary could be fractal (such as blood vessels, the internal structure of lungs, models of the stock market, etc),  is a natural generalization of Lipschitz domain with a small Lipschitz constant $\kappa>0$. In global regularity theory, the main obstruction is how to obtain regularity up to the boundary, and the geometric property of the domain near the boundary is important. The approach with $(\kappa, r_0)$-Reifenberg condition, after the work of Byun-Wang, inspired several subsequent developments in this line of research. In the spirit of assumption $\mathbf{(H_3)}$, it is worth mentioning that the Reifenberg flatness is meaningful for a certain range of $\kappa$, and in this work, we assume that $0<\kappa<\frac{1}{8}$.
\end{remark}

\begin{definition}[Fractional maximal functions]
\label{def:M_beta}
For a given $\beta \in [0, n)$, the maximal fractional function, usually written by $\mathbf{M}_\beta$, is a measurable map defined on $L^1_{\mathrm{loc}}(\mathbb{R}^n)$ as follows
\begin{align}\label{Malpha}
\mathbf{M}_\beta {f}(x) = \sup_{\rho>0} \rho^{\beta} \fint_{B_\rho(x)} |{f}(z)| dz, 
\end{align}
for all $x \in \mathbb{R}^n$ and ${f} \in L^1_{\mathrm{loc}}(\mathbb{R}^n)$. Corresponding to $\mathbf{M}_\beta$, let us also further define two ``cut-off'' fractional maximal functions of order $R>0$ as follows
\begin{align*}
\mathbf{M}_\beta^R f(x) = \sup_{0<\rho<R}{\rho^\beta \fint_{B_\rho(x)}{|f(z)| dz}}, \\
\mathbf{T}_\beta^R f(x) = \sup_{\rho\ge R}{\rho^\beta \fint_{B_\rho(x)}{|f(z)| dz}}.
\end{align*}
\end{definition}
\begin{proposition}[Boundedness of fractional maximal operators in Lebesgue spaces]
\label{lem:bound-M-beta}
For given $s \ge 1$, $\beta \in \left[0,\frac{n}{s}\right)$ and $f \in L^{s}(\mathbb{R}^n)$. Then, there exists a constant $C=C(n,s,\beta)>0$ such that for any $\lambda>0$, one has
\begin{align*}
\left|\left\{x \in \mathbb{R}^n: \ \mathbf{M}_{\beta}f(x)>\lambda\right\}\right| \le C \left(\lambda^{-s}\int_{\mathbb{R}^n}|f(z)|^{s} dz\right)^{\frac{n}{n-\beta s}}.
\end{align*}
\end{proposition}
We refer the interested reader to~\cite{PNJDE} for the detailed proof of this proposition. 
\begin{remark}
\label{rem:M_beta}
$ $ 
\begin{itemize}
\item[-] It is noteworthy to mention that when $\beta=0$, $\mathbf{M}_\beta$ coincides with $\mathbf{M}$, the Hardy-Littlewood maximal function. 
\item[-] One can remark in Definition~\ref{def:M_beta}, the case when $\beta=n$ is not included. Indeed, if $\beta=n$, from~\eqref{Malpha}, one observes that
\begin{align*}
\mathbf{M}_n {f}(x) = C(n)\sup_{\rho>0}{\int_{B_\rho(x)}|f(z)|dz},
\end{align*}
and it is no longer meaningful due to the unboundedness property (i.e. $\mathbf{M}_nf(x)$ can blow up even for $f \in L^1_{\mathrm{loc}}(\mathbb{R}^n)$). 
\item[-] Our work here shows an interest in global gradient estimates for weak solutions to~\eqref{OP-var} involving $\mathbf{M}_\beta$, that might bring a new challenge. The role of $\mathbf{M}_\beta$ was first indicated in~\cite{KM2014, Min2010, Min2011} with significant properties in regularity theory. The study of regularity estimates involving fractional maximal operators allows us to: (1) control the (local) oscillation of solutions; (2) estimate the upper bounds of fractional derivatives $\partial^\beta$ of solutions even when the classical derivatives do not exist; (3) understand the connection of $\mathbf{M}_\beta$ with the usual fractional Sobolev spaces $W^{\beta,q}$; (4) clarify the bridge between the \emph{fractional derivative} (in $L^q$-sense) and the fractional integral (Riesz potential) of order $\beta$ when $\beta$ is not too large. Here, it is hoped that interested readers will pay attention to~\cite{H1995, Adam1975, DS1984, NP2011, PNJFA, PNJDE, NTMN2025, PNJGA, PN-manu, TN2025_NA} for the picture of how the fractional maximal functions govern the regularity of solutions.
\end{itemize}
\end{remark}

\begin{definition}[Young function]\label{def:young}
Let $\mathcal{H}: [0,\infty) \to [0,\infty)$ be a non-decreasing convex function satisfying $\mathcal{H}(0)=0$. We say that $\mathcal{H}$ is a Young function if
\begin{align*}
\lim_{s \nearrow \infty} \frac{\mathcal{H}(s)}{s} = \infty, \ \mbox{ and } \ \lim_{s \searrow 0^+} \frac{\mathcal{H}(s)}{s} = 0.
\end{align*}
If there exists two constants $C_1, C_2>1$ such that 
\begin{align*}
\mathcal{H}(2s) \le C_1 \mathcal{H}(s) \, \ \mbox{ and } \ \ \mathcal{H}(s) \le \frac{\mathcal{H}(C_2s)}{2C_2} , \quad \text{for all}\ \  s \ge 0,
\end{align*}
then we shall often write $\mathcal{H} \in \Delta_2 \cap \nabla_2$.
\end{definition}

\begin{definition}[Weighted Orlicz spaces]
\label{def:orlicz}
Let $\mathcal{H} \in \Delta_2 \cap \nabla_2$ be a Young function and $\mu \in \mathcal{A}_{\infty}$. We denote by $L^{\mathcal{H}}_{\mu}(\Omega)$ the Orlicz space of $f \in \mathcal{M}eas(\Omega)$ satisfying 
$$\mathrm{mod}^{\mathcal{H}}_{\mu}(f) := \int_\Omega{\mathcal{H}(|f(x)|)d\mu} < \infty,$$ 
and endowed with the Luxemburg norm
\begin{align*}
\|f\|_{L^{\mathcal{H}_{\mu}}(\Omega)} = \inf\left\{s >0 : \ \mathrm{mod}^{\mathcal{H}}_{\mu}\left(\frac{f}{s} \right) \le 1 \right\}.
\end{align*}
\end{definition}

\textit{2.2. Some preliminary lemmas.} In the remaining parts of this section, we are devoted to a few preliminary lemmas that are useful for our needs later.

\begin{lemma}\label{lem:well-known}
Let $1<q<\infty$ and $p(\cdot)$ be a continuous function satisfying the assumption $\mathbf{(H_1)}$. Then, there exists a small constant $\kappa = \kappa(q,p_{\mathrm{min}},p_{\mathrm{max}})>0$ such that if $|\mathrm{log} \mathbb{W}|_{\mathrm{BMO}} \le \kappa$ then $\omega^{p(\cdot)} \in \mathcal{A}_{p(\cdot)}$-Muckenhoupt class. Moreover, there holds
\begin{align}\label{well-known-1}
\sup_{B} \left(\fint_{\Omega \cap B}\omega^qdx\right)^{\frac{1}{q}} \left(\fint_{\Omega \cap B}\omega^{q'}dx\right)^{\frac{1}{q'}} \le C(n,q,\Lambda),
\end{align}
where $q' = \displaystyle{\frac{q}{q-1}}$, and
\begin{align}\label{well-known-2}
\sup_{B} \left(\fint_{\Omega \cap B} \left|\frac{\mathbb{W}(x) - \langle\mathbb{W}\rangle_B^{\log}}{\langle\mathbb{W}\rangle_B^{\log}}\right|^{q} dx\right)^{\frac{1}{q}} \le C(n,q,\Lambda) |\mathrm{log} \mathbb{W}|_{\mathrm{BMO}},
\end{align}
where the supremum is taken over all balls $B \subset \mathbb{R}^n$.
\end{lemma}

The following covering argument, which has been developed in~\cite{CC1995, AM2007}, will play an essential role later in our proofs. 

\begin{lemma}\label{lem:Vitali}
Consider $\varepsilon \in (0,1)$ and a $(\kappa, r_0)$-Reifenberg domain $\Omega$ for some $\kappa, r_0>0$. Let $V_1$ and $V_2$ be two measurable subsets in $\Omega$. Assume that $|V_1| \le \varepsilon R_0^n$ for some $0<R_0\le r_0$. Suppose further that the following statement holds
\begin{align*}
\Omega \cap B_{\rho}(x_0) \not\subset V_2 \Longrightarrow |B_{\rho}(x_0)\cap V_1| < \varepsilon |B_{\rho}(x_0)|, \ \mbox{ for all } x_0 \in \Omega \mbox{ and } 0<\rho \le R_0.
\end{align*}
Then, there exists $C>0$ such that $|V_1| \leq C\varepsilon |V_2|$.
\end{lemma}

%
%
%

\section{Comparison maps}
\label{sec:comparison}

In this section, we establish the comparison strategy, which is divided into several steps towards the proofs of the main results. It is worth mentioning that one of the major challenges in the comparison procedure is due, on the one hand, to the structure of problem~\eqref{OP-var} that involves a singular/degenerate matrix weight $\mathbb{W}$ (the behavior of solutions is highly dependent on the properties of weights), on the other hand, to the variable exponent growths (variable function spaces in the weighted settings). 

More precisely, the process is divided into four steps associated with different comparison maps, and for the reader's convenience, we shall first sketch the main ideas and key points underlying the approach. We start from the comparison between the weak solution $u$ to the original problem~\eqref{OP-var} and the solution $v$ to a corresponding homogeneous equation, which will be stated in Lemma~\ref{lem-step1}. Subsequently, in the second step, it allows us to establish a local comparison estimate between the solution $v$ (to the homogeneous problem in the first step) and the solution $\tilde{v}$ to the new problem (obtained by fixing the variable exponent as $p_2$- the supremum of $p(\cdot)$ on a small ball). Following this idea, as a consequence, it yields that $\tilde{v}$ has higher integrability than $v$ on the ball with a smaller radius, via the assertion of Lemma~\ref{lem-step2}. In the next step, the goal is to further shrink the ball to one with a smaller radius and construct the new candidate $\pi$, solutions to the homogeneous equation whose matrix-weight is fixed by its logarithmic mean. Final preparation for the series of comparisons is devoted to the presence of $\tilde{\pi}$, which is bounded in $L^\infty$, the comparison between $\nabla\pi$ and $\nabla \tilde{\pi}$ will be established in Lemma~\ref{lem-step4}, even for the near-boundary balls. 

As a preliminary step, an initial local comparison estimate between the main solution to~\eqref{OP-var} and the solution to the associated homogeneous problem, via the following lemma. 
\begin{lemma}[The 1$^\text{st}$ step]
\label{lem-step1}
Let $u \in \mathbb{K}^\omega_{\phi_1,\phi_2}$ be a weak solution to the main problem~\eqref{OP-var} under assumptions \eqref{Pi-2},  \eqref{cond:pq1} and \eqref{def-phi-12}. There exists a constant $\kappa \in (0,1)$ such that if $\mathbf{(H_1)}$ satisfies for some $r_0>0$ then the following holds. Let $x_0$ be a fixed point in $\overline{\Omega}$, a given $R>0$, $p_2 = \sup\limits_{x \in \Omega_{R}} p(x)$ with $\Omega_R := \Omega_R(x_0)$. Then, it is possible to find a function $v \in W^{1,p(\cdot)}\big(\Omega_{2R},\omega^{p(\cdot)}\big) \cap W^{1,p_2}\left(\Omega_{R},\omega^{p_2}\right)$ such that
\begin{align}\label{est-step1a}
\fint_{\Omega_{2R}} \omega^{p(x)}  |\nabla u - \nabla v|^{p(x)} & dx \le \delta \fint_{\Omega_{2R}} \omega^{p(x)}|\nabla u|^{p(x)} dx \notag \\
& + C_{\delta} \fint_{\Omega_{2R}} \omega^{p(x)} \big(|\mathbf{F}|^{p(x)} + |\nabla \phi_1|^{p(x)} + |\nabla \phi_2|^{p(x)}\big) dx,
\end{align}
for every $\delta \in (0,1)$. Specifically, one obtains the following global estimate
\begin{align}\label{glob-est}
\int_{\Omega} \omega^{p(x)}|\nabla u|^{p(x)} dx \le C \int_{\Omega} \omega^{p(x)} \big(|\mathbf{F}|^{p(x)} + |\nabla \phi_1|^{p(x)} + |\nabla \phi_2|^{p(x)}\big) dx.
\end{align}
\end{lemma}
\begin{proof}
The proof is divided into several steps. We start here considering the model problem with a one-sided obstacle. Let us define
$$u_1 \in \mathrm{K}_{\phi_1} := \left\{\mathrm{w} \in u + W_0^{1,p(\cdot)}\big(\Omega_{2R},\omega^{p(\cdot)}\big): \ \mathrm{w} \ge \phi_1 \ \mbox{ a.e. in } \Omega_{2R}\right\},$$  
the unique solution to the following variational inequality
\begin{align}\label{eq:DOP-1}
\int_{\Omega_{2R}}  |\mathbb{W}(x)\nabla u_1|^{p(x)-2} & \mathbb{W}(x)\nabla u_1 \cdot \mathbb{W}(x)\nabla (u_1 -  \phi)  dx \notag \\
& \le \int_{\Omega_{2R}} |\mathbb{W}(x)\nabla \phi_2|^{p(x)-2}\mathbb{W}(x)\nabla \phi_2 \cdot \mathbb{W}(x)\nabla (u_1 -  \phi) dx,
\end{align}
for any $\phi \in \mathrm{K}_{\phi_1}$. Testing the formulation~\eqref{eq:DOP-1} with the function $u$, one gets that
\begin{align}\label{Comp-est-1}
\int_{\Omega_{2R}}   |\mathbb{W}(x)\nabla u_1|^{p(x)-2} & \mathbb{W}(x)\nabla u_1 \cdot \mathbb{W}(x)\nabla (u_1 -  u)  dx \notag \\
& \le \int_{\Omega_{2R}} |\mathbb{W}(x)\nabla \phi_2|^{p(x)-2}\mathbb{W}(x)\nabla \phi_2 \cdot \mathbb{W}(x)\nabla (u_1 -  u) dx.
\end{align}
Invoking~\eqref{Pi-2} and~\eqref{cond:pq1}, the estimate~\eqref{Comp-est-1} easily yields
\begin{align}\notag
 \fint_{\Omega_{2R}} & \omega^{p(x)}|\nabla u_1|^{p(x)} dx   \le C\fint_{\Omega_{2R}} \omega^{p(x)}|\nabla u_1|^{p(x)-1} |\nabla u| dx \\
& \qquad + C \fint_{\Omega_{2R}}  \omega^{p(x)}|\nabla \phi_2|^{p(x)-1} |\nabla u_1| dx + C \fint_{\Omega_{2R}}  \omega^{p(x)}|\nabla \phi_2|^{p(x)-1} |\nabla u| dx. \notag
\end{align}
Apply Young's inequality on the right-hand side of this inequality, we arrive at
\begin{align}\label{Comp-est-2}
& \fint_{\Omega_{2R}} \omega^{p(x)}|\nabla u_1|^{p(x)} dx   \le C\fint_{\Omega_{2R}} \omega^{p(x)}|\nabla u|^{p(x)} dx + C\fint_{\Omega_{2R}} \omega^{p(x)}|\nabla \phi_2|^{p(x)} dx.
\end{align}
On the other hand, choosing the test function for the test function $\phi$ in~\eqref{eq:DOP-1} 
$$u_1 - (u_1-\phi_2)^+ = \min\{u_1,\phi_2\} \in \mathrm{K}_{\phi_1},$$ 
we then obtain the inequality
\begin{align}\label{Comp-est-2-k}
\int_{\Omega_{2R} \cap \{u_1 \ge \phi_2\}} & \mathcal{J}_{p(x)}(\mathbb{W}(x)\nabla u_1,\mathbb{W}(x)\nabla \phi_2) dx \le 0,
\end{align} 
where $\mathcal{J}_{q}: \mathbb{R}^n \times \mathbb{R}^n \to [0,\infty)$ is the map defined by
\begin{align}\label{def-Jpx}
\mathcal{J}_{q}(y,z) = \left(|y|^{q-2}y - |z|^{q-2}z\right) \cdot (y-z), \quad \text{for} \ y,z \in \mathbb{R}^n.
\end{align}
For every $x \in \Omega$, regarding $\mathcal{J}_{p(x)}$, we further have the following classical inequality
\begin{align}\label{classi-ineq}
|y-z|^{p(x)} \le \epsilon |y|^{p(x)} + C_{\epsilon} \mathcal{J}_{p(x)}(y,z), \quad x \in \Omega,\quad y, z \in \mathbb{R}^n,
\end{align}
for all $\epsilon \in (0,1)$. Here, $C_{\epsilon}>0$ is a positive constant depending on $\epsilon$,  $p_{\mathrm{min}}$ and $p_{\mathrm{max}}$. It is further straightforward to verify~\eqref{classi-ineq}, that is somewhat similar to the one from~\cite[Lemma 3.2]{PN-manu} dealing with double-phase operators. As a consequence, one gets
\begin{align*}
\int_{\Omega_{2R} \cap \{u_1 \ge \phi_2\}} & \omega^{p(x)}|\nabla u_1 - \nabla \phi_2|^{p(x)} dx    \le \epsilon  \int_{\Omega_{2R} \cap \{u_1 \ge \phi_2\}} \omega^{p(x)}|\nabla u_1|^{p(x)} dx \\
& + C_{\epsilon} \int_{\Omega_{2R} \cap \{u_1 \ge \phi_2\}} \mathcal{J}_{p(x)}(\mathbb{W}(x)\nabla u_1,\mathbb{W}(x)\nabla \phi_2) dx,
\end{align*}
and combining with~\eqref{Comp-est-2-k}, it yields
\begin{align}\label{vphi-2}
\int_{\Omega_{2R} \cap \{u_1 \ge \phi_2\}} & \omega^{p(x)}|\nabla u_1 - \nabla \phi_2|^{p(x)} dx    \le \epsilon  \int_{\Omega_{2R} \cap \{u_1 \ge \phi_2\}} \omega^{p(x)}|\nabla u_1|^{p(x)} dx.
\end{align}
It is enough to substitute~\eqref{Comp-est-2} into~\eqref{vphi-2}, we arrive at
\begin{align*}
\int_{\Omega_{2R} \cap \{u_1 \ge \phi_2\}} \omega^{p(x)}|\nabla u_1 - \nabla \phi_2|^{p(x)} dx    \le \epsilon & \int_{\Omega} \omega^{p(x)} \Big(|\nabla u|^{p(x)} + |\nabla \phi_2|^{p(x)} \Big) dx,
\end{align*}
for $\epsilon>0$ small enough. As $\epsilon \to 0^+$, we infer that $u_1 \le \phi_2$ almost everywhere in $\Omega_{2R}$. Furthermore, we extend $u_1$ by $u$ outside of $\Omega_{2R}$, it therefore establishes $u_1$ as an admissible test function for~\eqref{OP-var}, from which we derive
\begin{align}\label{Comp-est-3}
\int_{\Omega_{2R}} & |\mathbb{W}(x)\nabla u|^{p(x)-2}  \mathbb{W}(x)\nabla u \cdot \mathbb{W}(x)\nabla (u -  u_1) dx \notag \\
& \qquad \le \int_{\Omega_{2R}} |\mathbb{W}(x) \mathbf{F}|^{p(x)-2}  \mathbb{W}(x) \mathbf{F} \cdot \mathbb{W}(x)\nabla (u -  u_1) dx.
\end{align}
By adding up~\eqref{Comp-est-1} and~\eqref{Comp-est-3}, it gets
\begin{align}\notag
\int_{\Omega_{2R}} & \mathcal{J}_{p(x)}(\mathbb{W}(x)\nabla u,\mathbb{W}(x)\nabla u_1) dx \le \int_{\Omega_{2R}} |\mathbb{W}(x) \mathbf{F}|^{p(x)-2}  \mathbb{W}(x) \mathbf{F} \cdot \mathbb{W}(x)\nabla (u -  u_1) dx \\
& \qquad \qquad + \int_{\Omega_{2R}} |\mathbb{W}(x)\nabla \phi_2|^{p(x)-2}\mathbb{W}(x)\nabla \phi_2 \cdot \mathbb{W}(x)\nabla (u_1 -  u) dx. \notag
\end{align}
At this stage, it observes that from~\eqref{Pi-2} and~\eqref{cond:pq1}, we infer
\begin{align}\notag
\fint_{\Omega_{2R}} \mathcal{J}_{p(x)}(\mathbb{W}(x)\nabla u,\mathbb{W}(x)\nabla u_1) dx   & \le C\fint_{\Omega_{2R}} \omega^{p(x)}|\mathbf{F}|^{p(x)-1} |\nabla u-\nabla u_1| dx \\
& \qquad + C\fint_{\Omega_{2R}} \omega^{p(x)}|\nabla \phi_2|^{p(x)-1} |\nabla u-\nabla u_1| dx, \notag
\end{align}
and making use of Young's inequality, we claim that
\begin{align}\notag
\fint_{\Omega_{2R}} \mathcal{J}_{p(x)}(\mathbb{W}(x)\nabla u,\mathbb{W}(x)\nabla u_1) dx   & \le \epsilon \fint_{\Omega_{2R}} \omega^{p(x)}|\nabla u-\nabla u_1|^{p(x)} dx \\
& \quad + C_{\epsilon} \fint_{\Omega_{2R}} \omega^{p(x)}|\mathbf{F}|^{p(x)} +  \omega^{p(x)}|\nabla \phi_2|^{p(x)} dx, \notag
\end{align}
for all $\epsilon \in (0,1)$. This together with~\eqref{classi-ineq} applied for any $\delta \in (0,1)$, it yields
\begin{align}\label{Comp-est-4}
 \fint_{\Omega_{2R}} \omega^{p(x)}|\nabla u - \nabla u_1|^{p(x)} dx   & \le \frac{\delta}{2} \fint_{\Omega_{2R}} \omega^{p(x)}|\nabla u|^{p(x)} dx \notag \\
 & \quad + \tilde{C}_{\delta} \epsilon \fint_{\Omega_{2R}} \omega^{p(x)}|\nabla u-\nabla u_1|^{p(x)} dx \notag \\
& \qquad + \tilde{C}_{\delta} C_{\epsilon} \fint_{\Omega_{2R}} \omega^{p(x)}|\mathbf{F}|^{p(x)} +  \omega^{p(x)}|\nabla \phi_2|^{p(x)} dx. 
\end{align}
As a result, it allows us to choose a suitable $\epsilon>0$ such that $\tilde{C}_{\delta} \epsilon \le \frac{1}{2}$ in~\eqref{Comp-est-4} to conclude the first step in this stage:
\begin{align}\label{Comp-est-5}
\fint_{\Omega_{2R}} \omega^{p(x)}|\nabla u - \nabla u_1|^{p(x)} dx   & \le \delta \fint_{\Omega_{2R}} \omega^{p(x)}|\nabla u|^{p(x)} dx \notag \\
& \qquad + C_{\delta} \fint_{\Omega_{2R}} \omega^{p(x)}|\mathbf{F}|^{p(x)} +  \omega^{p(x)}|\nabla \phi_2|^{p(x)} dx. 
\end{align}

To handle the next step regarding our assertion in Lemma~\ref{lem-step1}, we turn our attention to a homogeneous equation involving the given obstacle function $\phi_1$. More precisely, let us consider $u_2 \in u_1 + W_0^{1,p(\cdot)}\big(\Omega_{2R},\omega^{p(\cdot)}\big)$, a unique solution to the following variational formulation
\begin{align}\label{eq:DOP-2}
\int_{\Omega_{2R}}  |\mathbb{W}(x)\nabla u_2|^{p(x)-2} & \mathbb{W}(x)\nabla u_2 \cdot \mathbb{W}(x)\nabla \phi  dx \notag \\
& = \int_{\Omega_{2R}} |\mathbb{W}(x)\nabla \phi_1|^{p(x)-2}\mathbb{W}(x)\nabla \phi_1 \cdot \mathbb{W}(x)\nabla \phi dx,
\end{align}
for any $\phi \in W_0^{1,p(\cdot)}\big(\Omega_{2R},\omega^{p(\cdot)}\big)$. In analogy with the approach of the preceding step, it enables us to conclude $u_2 \ge \phi_1$ almost everywhere in $\Omega_{2R}$. Indeed, since~\eqref{eq:DOP-1} can be tested with $u_2$ and~\eqref{eq:DOP-2} with $u_2-u_1$, we arrive at
\begin{align}\notag
\int_{\Omega_{2R}} & \mathcal{J}_{p(x)}(\mathbb{W}(x)\nabla u_1,\mathbb{W}(x)\nabla u_2) dx \le \int_{\Omega_{2R}} |\mathbb{W}(x)\nabla \phi_1|^{p(x)-2}\mathbb{W}(x)\nabla \phi_1 \cdot \mathbb{W}(x)\nabla (u_2-u_1) dx \\
&~~~~~~~~~ \qquad \qquad + \int_{\Omega_{2R}} |\mathbb{W}(x)\nabla \phi_2|^{p(x)-2}\mathbb{W}(x)\nabla \phi_2 \cdot \mathbb{W}(x)\nabla (u_1 - u_2) dx. \notag
\end{align}
The proof follows along the same lines as those of~\eqref{Comp-est-5}, on making use of ~\eqref{Comp-est-2}, it is readily to verify that
\begin{align}\label{Comp-est-6}
\fint_{\Omega_{2R}} \omega^{p(x)}|\nabla u_1 - \nabla u_2|^{p(x)} dx   & \le \delta \fint_{\Omega_{2R}} \omega^{p(x)}|\nabla u|^{p(x)} dx \notag \\
& \quad + C_{\delta} \fint_{\Omega_{2R}} \omega^{p(x)}|\nabla \phi_1|^{p(x)} +  \omega^{p(x)}|\nabla \phi_2|^{p(x)} dx,
\end{align}
for every $\delta \in (0,1)$.

In the last step, it remains to derive the comparison estimate between $u_2$ and $v$, where $v \in u_2 + W_0^{1,p(\cdot)}\big(\Omega_{2R},\omega^{p(\cdot)}\big)$ is the unique solution to 
\begin{align}\label{eq:DOP-3}
\int_{\Omega_{2R}}  |\mathbb{W}(x)\nabla v|^{p(x)-2} & \mathbb{W}(x)\nabla v \cdot \mathbb{W}(x)\nabla \phi  dx = 0, \quad \forall \phi \in W_0^{1,p(\cdot)}\big(\Omega_{2R},\omega^{p(\cdot)}\big).
\end{align}
Follow a similar path, it yields that \begin{align}\label{Comp-est-7}
\fint_{\Omega_{2R}} \omega^{p(x)}|\nabla u_2 - \nabla v|^{p(x)} dx   & \le \delta \fint_{\Omega_{2R}} \omega^{p(x)}|\nabla u_2|^{p(x)} dx \notag \\
& \qquad + C_{\delta} \fint_{\Omega_{2R}} \omega^{p(x)}|\nabla \phi_1|^{p(x)} dx. 
\end{align}
Moreover, by standard computation, we infer that\begin{align}
\fint_{\Omega_{2R}} \omega^{p(x)}|\nabla v|^{p(x)} dx &\le C\fint_{\Omega_{2R}} \omega^{p(x)}|\nabla u_2|^{p(x)} dx \notag \\
& \le C\fint_{\Omega_{2R}} \omega^{p(x)} \big(|\nabla u_1|^{p(x)} + |\nabla \phi_1|^{p(x)}\big) dx \notag \\
& \le C\fint_{\Omega_{2R}} \omega^{p(x)} \big(|\nabla u|^{p(x)} + |\nabla \phi_1|^{p(x)} +  |\nabla \phi_2|^{p(x)}\big) dx. \label{step1b-3}
\end{align}
On the other hand, thanks to~\cite[Lemma 3.5]{BY24}, there exists a constant $\kappa>0$ such that
\begin{align}\label{step1b-0}
\fint_{\Omega_{R}} \left[\omega^{p(x)}|\nabla v|^{p(x)}\right]^{1+\epsilon} dx   & \le C \left[ \left(\fint_{\Omega_{2R}} \omega^{p(x)}|\nabla v|^{p(x)} dx\right)^{1 + \epsilon} + 1 \right],
\end{align}
for all $\epsilon \in (0,\kappa]$. It suffices to show that $v \in W^{1,p_2}\left(\Omega_{R},\omega^{p_2}\right)$ by using the $\log$-H\"older condition of function $p(\cdot)$ in~\eqref{cond:ap} and the higher integrability of $\nabla v$ given in~\eqref{step1b-0}. Indeed, thanks to~\eqref{cond:ap}, one obtains
\begin{align}\notag
p_2 \le p(x) + \alpha(4R) \le p(x) (1 + \alpha(4R)) \le p(x) (1 + \kappa),
\end{align}
and with a combination with~\eqref{step1b-0}, it leads to 
\begin{align}\label{step1b-1}
\fint_{\Omega_{R}} \omega^{p_2}  |\nabla v|^{p_2}  dx & \le \fint_{\Omega_{R}} (\omega^{p(x) (1 + \alpha(4R))}|\nabla u|^{p(x) (1 + \alpha(4R))} + 1) dx \notag \\
& \le C \left[\left(\fint_{\Omega_{2R}} \omega^{p(x)}|\nabla v|^{p(x)} dx\right)^{1 + \alpha(4R)} + 1\right] \notag \\
& \le C \left[\fint_{\Omega_{2R}} \omega^{p(x)}|\nabla v|^{p(x)} dx +1 \right] \left[\left(\fint_{\Omega_{2R}} \omega^{p(x)}|\nabla v|^{p(x)} dx\right)^{\alpha(4R)} + 1\right].
\end{align}
With the help of the global estimate~\eqref{glob-est} and the $\log$-H\"older condition~\eqref{cond:ap}, we claim that
\begin{align}\notag
\left(\fint_{\Omega_{2R}} \omega^{p(x)}|\nabla v|^{p(x)} dx\right)^{\alpha(4R)} &\le \left[ \frac{C}{(4R)^n} \int_{\Omega} \omega^{p(x)} \big(|\mathbf{F}|^{p(x)} + |\nabla \phi_1|^{p(x)} + |\nabla \phi_2|^{p(x)}\big) dx\right]^{\alpha(4R)} \\
& \le C \exp\left(\alpha(4R) \mathrm{log}\frac{1}{4R}\right) \notag \\
& \le C. \label{step1b-2}
\end{align}
It is worth mentioning that the constant $C$ here still depends on the data $\|\mathbf{F}\|_{L^{p(\cdot)}(\Omega,\omega^{p(\cdot)})}$, $\|\nabla \phi_1\|_{L^{p(\cdot)}(\Omega,\omega^{p(\cdot)})}$ and $\|\nabla \phi_2\|_{L^{p(\cdot)}(\Omega,\omega^{p(\cdot)})}$. Next, substituting~\eqref{step1b-3} and~\eqref{step1b-2} into~\eqref{step1b-1}, one obtains that
\begin{align}\label{est-step1b}
\fint_{\Omega_{R}} \omega^{p_2}  |\nabla v|^{p_2} & dx \le C \left(\fint_{\Omega_{2R}} \omega^{p(x)} \big(|\nabla u|^{p(x)} + |\nabla \phi_1|^{p(x)} +  |\nabla \phi_2|^{p(x)}\big) dx + 1\right),
\end{align}
which allows us to conclude $v \in W^{1,p_2}\left(\Omega_{R},\omega^{p_2}\right)$. 

Finally, we combine~\eqref{Comp-est-5},~\eqref{Comp-est-6} and~\eqref{Comp-est-7} to obtain~\eqref{est-step1a}. Consequently, the global estimate~\eqref{glob-est} will follow by choosing $R>0$ sufficiently large so that $\Omega \subset B_{2R}(x_0)$, which guarantees that the solution $v$ to~\eqref{eq:DOP-3} vanishes outside this ball. Hence,~\eqref{glob-est} can be concluded directly from~\eqref{est-step1a}. 
\end{proof}

The second step of the comparison strategy, presented in Lemma~\ref{lem-step2} below, is devoted to establishing the comparison estimate between $\nabla v$ and $\nabla \tilde{v}$. Therein, one considers a new homogeneous problem with constant $p_2$-growth, where $p_2 = \sup\limits_{x \in \Omega_{R}} p(x)$. 

\begin{lemma}[The 2$^\text{nd}$ step]
\label{lem-step2}
Under the assumptions of Lemma~\ref{lem-step1} and let $v \in W^{1,p(\cdot)}\big(\Omega_{2R},\omega^{p(\cdot)}\big) \cap W^{1,p_2}\left(\Omega_{R},\omega^{p_2}\right)$ be the solution to~\eqref{est-step1a}. Then, there exists $\tilde{v} \in v + W^{1,p_2}_{0}(\Omega_{R},\omega^{p_2})$ such that the following estimate holds true:
\begin{align}\label{est-step2.2}
\fint_{\Omega_{R}} \omega^{p_2} & |\nabla v - \nabla \tilde{v}|^{p_2} dx  \le \left(\delta + C_{\delta} \kappa^{\frac{p_2}{p_2-1}}\right) \fint_{\Omega_{2R}} \omega^{p(x)}|\nabla u|^{p(x)} dx \notag \\
& \qquad +   C_{\delta} \left[1 + \kappa^{\frac{p_2}{p_2-1}} \fint_{\Omega_{2R}} \omega^{p(x)} \big(|\nabla \phi_1|^{p(x)} + |\nabla \phi_2|^{p(x)}\big) dx \right],
\end{align}
for every $\delta \in (0,1)$. 
\end{lemma}
\begin{proof}
Let us consider $\tilde{v} \in v + W^{1,p_2}_{0}(\Omega_{R},\omega^{p_2})$ the unique solution to the variational formulation
\begin{align}\label{Prob-4}
\int_{\Omega_{R}}  |\mathbb{W}(x)\nabla \tilde{v}|^{p_2-2} & \mathbb{W}(x)\nabla \tilde{v} \cdot \mathbb{W}(x)\nabla \phi  dx = 0, \quad \text{for any} \ \phi \in W^{1,p_2}_{0}(\Omega_{R},\omega^{p_2}). 
\end{align}
By the standard estimate in $L^{p_2}(\Omega_{R},\omega^{p_2})$ for $\tilde{v}$, it is easily to show that
\begin{align}\label{Prob-4-est-1}
\fint_{\Omega_{R}} \omega^{p_2}|\nabla \tilde{v}|^{p_2} dx   & \le C\fint_{\Omega_{R}} \omega^{p_2}|\nabla v|^{p_2} dx. 
\end{align}
Moreover, there exists a constant $\epsilon_1 \in (0, \kappa]$ such that
\begin{align}\label{RH-v}
\left[\fint_{\Omega_{R/2}} \left(\omega^{p_2}|\nabla \tilde{v}|^{p_2}\right)^{1+\epsilon} dx\right]^{\frac{1}{1+\epsilon}}   & \le C \left(\fint_{\Omega_{R}} \omega^{p_2}|\nabla \tilde{v}|^{p_2} dx + 1 \right),
\end{align}
for all $\epsilon \in (0,\epsilon_1]$. Let us test~\eqref{eq:DOP-3} and~\eqref{Prob-4} by $v - \tilde{v} \in W^{1,p_2}_{0}(\Omega_{R},\omega^{p_2})$ to arrive
\begin{align}
\int_{\Omega_{R}} & \mathcal{J}_{p_2}(\mathbb{W}(x)\nabla v,\mathbb{W}(x)\nabla \tilde{v}) dx \notag \\
&  = \int_{\Omega_{R}} \left[|\mathbb{W}(x)\nabla v|^{p_2-2} \mathbb{W}(x)\nabla v -  |\mathbb{W}(x)\nabla v|^{p(x)-2}  \mathbb{W}(x)\nabla v\right] \cdot \mathbb{W}(x)\nabla (v - \tilde{v})  dx, \notag
\end{align}
where $\mathcal{J}_{p_2}$ is defined as in~\eqref{def-Jpx}. To estimate the right-hand side, it allows us to apply the following fundamental inequality 
\begin{align*}
\left||y|^{p_2-1}-|y|^{p(x)-1}\right| \le C \alpha(2R)\left[1 + |y|^{p_2-1}\mathrm{log}(e + |y|^{p_2})\right], \quad y\in \mathbb{R}^n,
\end{align*}
from which, the interested reader may see~\cite[Lemma 3.14]{PNJGA} for a quite detailed proof. Since then, it follows that
\begin{align}
\fint_{\Omega_{R}} & \mathcal{J}_{p_2}(\mathbb{W}(x)\nabla v,\mathbb{W}(x)\nabla \tilde{v}) dx \notag \\ 
 &\qquad  \le C  \fint_{\Omega_{R}} \alpha(2R)\left[1 + \omega^{p_2-1} |\nabla v|^{p_2-1}\mathrm{log}(e + \omega^{p_2}|\nabla v|^{p_2})\right] \omega\left|\nabla v - \nabla \tilde{v}\right|  dx.\notag
\end{align}
In a similar fashion to Lemma~\ref{lem-step1},  for every $\delta \in (0,1)$, there holds\begin{align}
\fint_{\Omega_{R}} \omega^{p_2}|\nabla v - \nabla \tilde{v}|^{p_2} dx & \le \delta  \fint_{\Omega_{R}} \omega^{p_2}|\nabla v|^{p_2} dx + C_{\delta} \fint_{\Omega_{R}} \mathcal{J}_{p_2}(\mathbb{W}(x)\nabla v,\mathbb{W}(x)\nabla \tilde{v}) dx \notag \\
& \le \delta  \fint_{\Omega_{R}} \omega^{p_2}|\nabla v|^{p_2} dx \notag \\
& \qquad + C_{\delta}  \fint_{\Omega_{R}} \alpha(2R)\left[1 + \omega^{p_2-1} |\nabla v|^{p_2-1}\mathrm{log}(e + \omega^{p_2}|\nabla v|^{p_2})\right] \omega\left|\nabla v - \nabla \tilde{v}\right| dx.\notag
\end{align}
Thanks to Young's inequality, one obtains that
\begin{align}
\fint_{\Omega_{R}} & \omega^{p_2}|\nabla v - \nabla v|^{p_2} dx  \le \delta  \fint_{\Omega_{R}} \omega^{p_2}|\nabla v|^{p_2} dx \notag \\
& \qquad  + C_{\delta} \left[1 +  [\alpha(2R)]^{\frac{p_2}{p_2-1}} \fint_{\Omega_{R}} \omega^{p_2} |\nabla v|^{p_2}\mathrm{log}^{\frac{p_2}{p_2-1}}(e + \omega^{p_2}|\nabla v|^{p_2}) dx \right].\label{Prob-4-est-2}
\end{align}
Thanks to~\cite[Lemma 3.13]{PNJGA}, if provided $0<\displaystyle{\fint_{\Omega_{R}} \omega^{p_2}|\nabla v|^{p_2} dx} \le 1$, we infer that 
\begin{align}\notag
\fint_{\Omega_{R}} \omega^{p_2} |\nabla v|^{p_2}\mathrm{log}^{\frac{p_2}{p_2-1}} & \left(e + \omega^{p_2}|\nabla v|^{p_2}\right) dx \notag \\
& \le \fint_{\Omega_{R}} \omega^{p_2} |\nabla v|^{p_2}\mathrm{log}^{\frac{p_2}{p_2-1}}\left(e + \frac{\omega^{p_2}|\nabla v|^{p_2}}{\fint_{\Omega_{R}} \omega^{p_2}|\nabla v|^{p_2} dx}\right) dx \notag \\
& \le C\left[\fint_{\Omega_{R}} \left(\omega^{p_2} |\nabla v|^{p_2}\right)^{1+\epsilon} dx\right]^{\frac{1}{1+\epsilon}},\notag
\end{align}
for every $\epsilon>0$. By following the same lines as in the proof of~\eqref{step1b-1}, choosing $\epsilon>0$ such that $(1+\epsilon) (1 + \alpha(2R)) \le 1 + \kappa$ and applying~\eqref{step1b-0}, we deduce that
\begin{align}\notag
\left[\fint_{\Omega_{R}} \left(\omega^{p_2} |\nabla v|^{p_2}\right)^{1+\epsilon} dx\right]^{\frac{1}{1+\epsilon}} & \le \left[\fint_{\Omega_{R}} (\omega^{p(x)(1+\epsilon) (1 + \alpha(2R))}|\nabla v|^{p(x)(1+\epsilon)(1 + \alpha(2R))} + 1) dx\right]^{\frac{1}{1+\epsilon}} \notag \\
& \le C \left[\left(\fint_{\Omega_{2R}} \omega^{p(x)}|\nabla v|^{p(x)} dx\right)^{1 + \alpha(2R)} + 1\right] \notag \\
& \le C \left[\fint_{\Omega_{2R}} \omega^{p(x)}|\nabla v|^{p(x)} dx +1 \right] \left[\left(\fint_{\Omega_{2R}} \omega^{p(x)}|\nabla v|^{p(x)} dx\right)^{\alpha(2R)} + 1\right] \notag \\ 
& \le C \left[\fint_{\Omega_{2R}} \omega^{p(x)}|\nabla v|^{p(x)} dx +1 \right].\notag
\end{align}
Therefore, we arrive at
\begin{align}\label{Prob-4-est-3}
\fint_{\Omega_{R}} \omega^{p_2} |\nabla v|^{p_2}\mathrm{log}^{\frac{p_2}{p_2-1}}\left(e + \omega^{p_2}|\nabla v|^{p_2}\right) dx & \le C \left[\fint_{\Omega_{2R}} \omega^{p(x)}|\nabla v|^{p(x)} dx +1 \right].
\end{align}
At this stage, let us now consider the remaining case when $\fint_{\Omega_{R}} \omega^{p_2}|\nabla v|^{p_2} dx > 1$. We first apply the following elementary inequality
\begin{align*}
\mathrm{log}^{\beta}(e+ab) \le 2^{\beta}\left[\mathrm{log}^{\beta}(e+a) + \mathrm{log}^{\beta}(b)\right], \ \mbox{ for all } \ \beta>0, \ a>0, \ b>1,
\end{align*}
and taking~\eqref{Prob-4-est-3} into account, it yields
\begin{align}\notag
\fint_{\Omega_{R}} \omega^{p_2} |\nabla v|^{p_2} & \mathrm{log}^{\frac{p_2}{p_2-1}} \left(e + \omega^{p_2}|\nabla v|^{p_2}\right) dx \notag \\
& \le C\fint_{\Omega_{R}} \omega^{p_2} |\nabla v|^{p_2}\mathrm{log}^{\frac{p_2}{p_2-1}}\left(e + \frac{\omega^{p_2}|\nabla v|^{p_2}}{\fint_{\Omega_{R}} \omega^{p_2}|\nabla v|^{p_2} dx}\right) dx \notag \\
& \qquad \qquad + C\fint_{\Omega_{R}} \omega^{p_2} |\nabla v|^{p_2}\mathrm{log}^{\frac{p_2}{p_2-1}}\left(\fint_{\Omega_{R}} \omega^{p_2}|\nabla v|^{p_2} dx\right) dx\notag \\
& \le C\left[\fint_{\Omega_{2R}} \omega^{p(x)} |\nabla v|^{p(x)} dx + 1\right] \notag \\
& \qquad \qquad + C\mathrm{log}^{\frac{p_2}{p_2-1}}\left(\fint_{\Omega_{R}} \omega^{p_2}|\nabla v|^{p_2} dx\right) \fint_{\Omega_{R}} \omega^{p_2} |\nabla v|^{p_2} dx. \notag
\end{align}
Next, the logarithmic term on the right-hand side can be estimated by applying~\eqref{Prob-4-est-1} and~\eqref{est-step1b} as follows
\begin{align}
\mathrm{log}\left(\fint_{\Omega_{R}} \omega^{p_2}|\nabla v|^{p_2} dx\right) &\le C \left[\mathrm{log} \left(\frac{1}{R^n} \int_{\Omega} \omega^{p(x)} \left(|\mathbf{F}|^{p(x)} + |\nabla \phi_1|^{p(x)} +  |\nabla \phi_2|^{p(x)}\right) dx + 1\right) \right] \notag \\
&  \le C \left[\mathrm{log} \left(\frac{1}{2R}\right) + 1\right], \notag
\end{align}
and note that the constant $C>0$ here depends on $\|\mathbf{F}\|_{L^{p(\cdot)}(\Omega,\omega^{p(\cdot)})}$, $\|\nabla \phi_1\|_{L^{p(\cdot)}(\Omega,\omega^{p(\cdot)})}$ and $\|\nabla \phi_2\|_{L^{p(\cdot)}(\Omega,\omega^{p(\cdot)})}$. Then, reabsorbing this estimate and~\eqref{step1b-1} into the right-hand side to arrive at
\begin{align}
\fint_{\Omega_{R}} \omega^{p_2} |\nabla v|^{p_2} \mathrm{log}^{\frac{p_2}{p_2-1}} & \left(e + \omega^{p_2}|\nabla v|^{p_2}\right) dx \notag \\
& \le C\left[\fint_{\Omega_{2R}} \omega^{p(x)} |\nabla v|^{p(x)} dx + 1\right] \notag \\ 
& \qquad \qquad + C\left[\mathrm{log}^{\frac{p_2}{p_2-1}}\left(\frac{1}{2R}\right) \fint_{\Omega_{R}} \omega^{p_2} |\nabla v|^{p_2} dx + 1\right] \notag \\
& \le C\left[\mathrm{log}^{\frac{p_2}{p_2-1}}\left(\frac{1}{2R}\right) \fint_{\Omega_{2R}} \omega^{p(x)} |\nabla v|^{p(x)} dx + 1\right], \label{Prob-4-est-4}
\end{align}
Combining~\eqref{Prob-4-est-2},~\eqref{Prob-4-est-3} and~\eqref{Prob-4-est-4}, one obtains
\begin{align}
\fint_{\Omega_{R}} & \omega^{p_2}|\nabla v - \nabla \tilde{v}|^{p_2} dx  \le \delta  \fint_{\Omega_{R}} \omega^{p_2}|\nabla v|^{p_2} dx \notag \\
& \qquad  + C_{\delta} \left[[\alpha(2R)]^{\frac{p_2}{p_2-1}} \mathrm{log}^{\frac{p_2}{p_2-1}}\left(\frac{1}{2R}\right) \fint_{\Omega_{2R}} \omega^{p(x)} |\nabla v|^{p(x)} dx + 1 \right].\label{Prob-4-est-5}
\end{align}
Finally, we are at the point to conclude the assertion of~\eqref{est-step2.2} from~\eqref{Prob-4-est-5} by applying the $\log$-H\"older condition~\eqref{cond:ap} and two estimates~\eqref{step1b-3},~\eqref{est-step1b}.
\end{proof}

From Lemma~\ref{lem-step2}, it can be seen that solution $\tilde{v}$ exhibits higher integrability than $v$ on a concentric ball with a smaller radius. The next result provides us with another comparison level. Here, our objective is to further investigate a new homogeneous equation associated with the matrix-weight $\mathbb{W}_0$ (the logarithmic mean of $\mathbb{W}$), which admits a unique solution denoted by $\pi$. Lemma~\ref{lem-step3} will then establish the comparison between $\nabla \tilde{v}$ ($\tilde{v}$ is precribed in Lemma~\ref{lem-step2}) and $\nabla \pi$. It is worth noting that for this proof, the assumption $\mathbf{(H_2)}$ is required to guarantee the finiteness of the average integral of the weight function $\omega\omega_0^{-1}$, where $\omega_0$ denotes the logarithmic mean of $\omega$.

\begin{lemma}[The 3$^\text{rd}$ step]
\label{lem-step3}
Under the assumption of Lemma~\ref{lem-step2} and let $v \in W^{1,p(\cdot)}\big(\Omega_{2R},\omega^{p(\cdot)}\big) \cap W^{1,p_2}\left(\Omega_{R},\omega^{p_2}\right)$,  $\tilde{v} \in v + W^{1,p_2}_{0}(\Omega_{R},\omega^{p_2})$ be the solutions to~\eqref{est-step1a} and~\eqref{Prob-4}, respectively. Then, there exists $\pi \in \tilde{v} + W^{1,p_2}_{0}(\Omega_{R/2},\omega_0^{p_2})$ such that
\begin{align}\label{est-step3.2}
\fint_{\Omega_{R/2}}  \big|\omega & |\nabla\tilde{v}| - \omega_{0} |\nabla \pi|\big|^{p_2} dx  \le \left(\delta + C_{\delta} |\mathrm{log} \mathbb{W}|_{\mathrm{BMO}}\right) \fint_{\Omega_{2R}} \omega^{p(x)}|\nabla u|^{p(x)} dx \notag \\
& \qquad +   C_{\delta} |\mathrm{log} \mathbb{W}|_{\mathrm{BMO}} \left[1 +  \fint_{\Omega_{2R}} \omega^{p(x)} \big(|\nabla \phi_1|^{p(x)} + |\nabla \phi_2|^{p(x)}\big) dx \right],
\end{align}
for every $\delta \in (0,1)$, provided that $|\mathrm{log} \mathbb{W}|_{\mathrm{BMO}}$ small enough (described in Assumption $\mathbf{(H_2)}$). Here, for notational purposes, we further denote 
$$\mathbb{W}_0 = \langle \mathbb{W} \rangle_{B_{R/2}}^{\mathrm{log}}, \mbox{ and } \omega_0 = \langle \omega \rangle_{B_{R/2}}^{\mathrm{log}}.$$
\end{lemma}
\begin{proof}
First we show $\tilde{v} \in W^{1,p_2}_{0}(\Omega_{R/2},\omega_0^{p_2})$. Indeed, as the reader may easily check that $\Lambda^{-1} \omega_0 \le |\mathbb{W}_0| \le \omega_0$. Moreover, for every $q \ge 1$, one can find $\kappa>0$ small enough such that if $|\mathrm{log} \mathbb{W}|_{\mathrm{BMO}}<\kappa$,   then one has  
\begin{align}\label{w-w0-bound}
\fint_{\Omega_{R/2}} \left(\omega^q \omega_0^{-q} + \omega_0^q \omega^{-q}\right) dx \le C.
\end{align}
It is worthwhile to address that~\eqref{w-w0-bound} is a consequence of~\eqref{well-known-1} in Lemma~\ref{lem:well-known}. By applying H\"older's inequality for some $\epsilon \in (0,\epsilon_1]$, we observe that 
\begin{align}
\fint_{\Omega_{R/2}} \omega_0^{p_2}|\nabla \tilde{v}|^{p_2} dx & \le C \left[\fint_{\Omega_{R/2}} \left(\omega^{p_2}|\nabla \tilde{v}|^{p_2}\right)^{1+\epsilon} dx\right]^{\frac{1}{1+\epsilon}} \left[\fint_{\Omega_{R/2}} \left(\omega^{-p_2} \omega_0^{p_2}\right)^{\frac{\epsilon+1}{\epsilon}} dx\right]^{\frac{\epsilon}{\epsilon+1}}, \notag 
\end{align} 
and gathering~\eqref{RH-v} and~\eqref{w-w0-bound}, it leads to 
\begin{align}
\fint_{\Omega_{R/2}} \omega_0^{p_2}|\nabla \tilde{v}|^{p_2} dx & \le C \left[\fint_{\Omega_{R}} \omega^{p_2}|\nabla \tilde{v}|^{p_2} dx + 1\right], \label{Prob-5-est-0}
\end{align} 
which ensures that $\tilde{v} \in W^{1,p_2}_{0}(\Omega_{R/2},\omega_0^{p_2})$. At this stage, let us consider $\pi \in \tilde{v} + W^{1,p_2}_{0}(\Omega_{R/2},\omega_0^{p_2})$ the unique solution to 
\begin{align}\label{Prob-5}
\int_{\Omega_{R/2}}  |\mathbb{W}_0\nabla \pi|^{p_2-2} & \mathbb{W}_0\nabla \pi \cdot \mathbb{W}_0\nabla \phi  dx = 0, \quad \text{for any} \ \phi \in W^{1,p_2}_{0}(\Omega_{R/2},\omega_0^{p_2}).
\end{align}
The proof exploits an argument similar to previous problems. It therefore enables us to prove that there exists a constant $\epsilon_2 \in (0,\epsilon_1]$ such that
\begin{align}\label{RH-v5}
\left[\fint_{\Omega_{R/4}} \left(\omega_0^{p_2}|\nabla \pi|^{p_2}\right)^{1+\epsilon} dx\right]^{\frac{1}{1+\epsilon}}   & \le C \left(\fint_{\Omega_{R/2}} \omega_0^{p_2}|\nabla \pi|^{p_2} dx + 1 \right),
\end{align}
for all $\epsilon \in (0,\epsilon_2]$, and 
\begin{align}\label{Prob-5-est-1}
\fint_{\Omega_{R/2}} \omega_0^{p_2}|\nabla \pi|^{p_2} dx   & \le C\fint_{\Omega_{R/2}} \omega_0^{p_2}|\nabla \tilde{v}|^{p_2} dx. 
\end{align} 
It is important to note that inequalities~\eqref{Prob-5-est-0} and~\eqref{Prob-5-est-1} guarantee that  $\pi \in W^{1,p_2}(\Omega_{R/2},\omega_0^{p_2})$. Consequently, it allows us to take $\pi - \tilde{v} \in W^{1,p_2}_{0}(\Omega_{R/2},\omega_0^{p_2})$ as a test function in~\eqref{Prob-4} and~\eqref{Prob-5}, and thus we arrive at
\begin{align}
\int_{\Omega_{R/2}} \mathcal{J}_{p_2} \left(\mathbb{W}_0\nabla \pi,\mathbb{W}_0\nabla \tilde{v}\right)dx  & = \int_{\Omega_{R/2}} |\mathbb{W}(x)\nabla \tilde{v}|^{p_2-2} \mathbb{W}^2(x)\nabla \tilde{v} \cdot (\nabla \pi - \nabla\tilde{v}) dx \notag \\
 & \qquad - \int_{\Omega_{R/2}} |\mathbb{W}_0\nabla \tilde{v}|^{p_2-2} \mathbb{W}_0^2\nabla \tilde{v} \cdot (\nabla \pi - \nabla\tilde{v}) dx, \notag
\end{align}
where $\mathcal{J}_{p_2}$ is defined as in~\eqref{def-Jpx}. Thanks to~\eqref{classi-ineq}, one gets that
\begin{align}\label{Prob-5-est-2}
\fint_{\Omega_{R/2}} \omega_{0}^{p_2} \left|\nabla \pi - \nabla\tilde{v}\right|^{p_2} dx &\le \delta \fint_{\Omega_{R/2}} \omega_0^{p_2}|\nabla \tilde{v}|^{p_2} dx + C_{\delta} \fint_{\Omega_{R/2}} \mathcal{J}_{p_2} \left(\mathbb{W}_0\nabla \pi,\mathbb{W}_0\nabla \tilde{v}\right)dx \notag \\
& \le \delta \fint_{\Omega_{R/2}} \omega_0^{p_2}|\nabla \tilde{v}|^{p_2} dx + C_{\delta} \fint_{\Omega_{R/2}} |\mathcal{K}(x)| |\nabla \pi - \nabla\tilde{v}|dx,
\end{align}
for every $\delta \in (0,1)$, and the sake of exposition, we write
\begin{align}\notag
\mathcal{K}(x) := |\mathbb{W}(x)\nabla \tilde{v}|^{p_2-2} \mathbb{W}^2(x)\nabla \tilde{v}-|\mathbb{W}_0\nabla \tilde{v}|^{p_2-2} \mathbb{W}_0^2\nabla \tilde{v}.\notag
\end{align}
Moreover, to estimate $|\mathcal{K}(x)|$, we first analyze as follows
\begin{align}\notag
|\mathcal{K}(x)| & \le \left||\mathbb{W}(x)\nabla \tilde{v}|^{p_2-2} \mathbb{W}^2(x)\nabla \tilde{v}-|\mathbb{W}(x)\nabla \tilde{v}|^{p_2-2} \mathbb{W}_0\mathbb{W}(x)\nabla \tilde{v}\right| \\
& \qquad + \left||\mathbb{W}(x)\nabla \tilde{v}|^{p_2-2} \mathbb{W}_0\mathbb{W}(x)\nabla \tilde{v}-|\mathbb{W}_0\nabla \tilde{v}|^{p_2-2} \mathbb{W}_0^2\nabla \tilde{v}\right| \notag \\
& \le |\mathbb{W}(x)-\mathbb{W}_0||\mathbb{W}(x)\nabla \tilde{v}|^{p_2-1} \notag \\
& \qquad + |\mathbb{W}_0|\left||\mathbb{W}(x)\nabla \tilde{v}|^{p_2-2} \mathbb{W}(x)\nabla \tilde{v}-|\mathbb{W}_0\nabla \tilde{v}|^{p_2-2} \mathbb{W}_0\nabla \tilde{v}\right| \notag \\
& \le C \left(|\mathbb{W}(x)|^{p_2} + |\mathbb{W}_0|^{p_2}\right) \frac{|\mathbb{W}(x) - \mathbb{W}_0|}{|\mathbb{W}(x)| + |\mathbb{W}_0|} |\nabla \tilde{v}|^{p_2-1} \notag \\
& \qquad + |\mathbb{W}_0|\left||\mathbb{W}(x)\nabla \tilde{v}|^{p_2-2} \mathbb{W}(x)\nabla \tilde{v}-|\mathbb{W}_0\nabla \tilde{v}|^{p_2-2} \mathbb{W}_0\nabla \tilde{v}\right|.\notag
\end{align}
For the last term in the preceding inequality, we write, using the mean value theorem for the continuous map $\zeta \in \mathbb{R}^n \mapsto |\zeta|^{p_2-2}\zeta$. Specifically, there holds
\begin{align}
|\mathbb{W}_0|\big||\mathbb{W}(x)\nabla \tilde{v}|^{p_2-2} & \mathbb{W}(x)\nabla \tilde{v}-|\mathbb{W}_0\nabla \tilde{v}|^{p_2-2} \mathbb{W}_0\nabla \tilde{v}\big| \notag \\
&\le C |\mathbb{W}_0|\left|\tau \mathbb{W}(x)\nabla \tilde{v} + (1-\tau) \mathbb{W}_0\nabla \tilde{v}\right|^{p_2-2} |\mathbb{W}(x) - \mathbb{W}_0||\nabla \tilde{v}| \notag \\
& \le C |\mathbb{W}_0|\frac{\left(|\mathbb{W}(x)| + |\mathbb{W}_0|\right)^{p_2}}{\left|\tau \mathbb{W}(x) + (1-\tau) \mathbb{W}_0\right|^{2}} |\mathbb{W}(x) - \mathbb{W}_0||\nabla \tilde{v}|^{p_2-1} \notag \\
& \le C  \left(|\mathbb{W}(x)| + |\mathbb{W}_0|\right)^{p_2}  \left(\frac{|\mathbb{W}_0|}{|\mathbb{W}(x)|} + 1\right) \frac{|\mathbb{W}(x) - \mathbb{W}_0|}{|\mathbb{W}(x)| + |\mathbb{W}_0|} |\nabla \tilde{v}|^{p_2-1}. \notag
\end{align}
And then, we obtain the following estimate
\begin{align}
|\mathcal{K}(x)| \le C \left(|\mathbb{W}(x)| + |\mathbb{W}_0|\right)^{p_2} \left|\frac{\mathbb{W}(x) - \mathbb{W}_0}{\mathbb{W}_0}\right| \left(\frac{|\mathbb{W}_0|}{|\mathbb{W}(x)|} + 1\right) |\nabla \tilde{v}|^{p_2-1}. \notag
\end{align}
Plugging this estimate into~\eqref{Prob-5-est-2} and then invoking Young's inequality, we infer
\begin{align}\notag
& \fint_{\Omega_{R/2}}  \omega_{0}^{p_2} \left|\nabla \pi - \nabla\tilde{v}\right|^{p_2} dx 
 \le \delta \fint_{\Omega_{R/2}} \omega_0^{p_2}|\nabla \tilde{v}|^{p_2} dx \notag \\
& \quad + C_{\delta} \fint_{\Omega_{R/2}} \left[\left(\frac{|\mathbb{W}(x)|}{|\mathbb{W}_0|} + 1\right)^{p_2} \left(\frac{|\mathbb{W}_0|}{|\mathbb{W}(x)|} + 1\right)^{p_2} \left|\frac{\mathbb{W}(x) - \mathbb{W}_0}{\mathbb{W}_0}\right| |\mathbb{W}(x)|^{p_2-1}|\nabla \tilde{v}|^{p_2-1}\right]^{\frac{p_2}{p_2-1}} dx.\notag
\end{align}
For $\epsilon \in (0,\epsilon_1]$, thanks to H\"older's inequality, one gets
\begin{align}\label{Prob-5-est-3}
\fint_{\Omega_{R/2}}  \omega_{0}^{p_2}  \left|\nabla \pi - \nabla\tilde{v}\right|^{p_2} dx & \le \delta \fint_{\Omega_{R/2}} \omega_0^{p_2}|\nabla \tilde{v}|^{p_2} dx \notag \\
& \qquad + C_{\delta} \mathcal{K}_1 \mathcal{K}_2 \left[\fint_{\Omega_{R/2}} \left(\omega^{p_2}|\nabla \tilde{v}|^{p_2}\right)^{1+\epsilon} dx\right]^{\frac{1}{1+\epsilon}}.
\end{align}
where $\mathcal{K}_1, \mathcal{K}_2$ are respectively defined by
\begin{align*}
\mathcal{K}_1 = \left[\fint_{\Omega_{R/2}} \left[\left(\frac{|\mathbb{W}(x)|}{|\mathbb{W}_0|} + 1\right)^{p_2} \left(\frac{|\mathbb{W}_0|}{|\mathbb{W}(x)|} + 1\right)^{p_2}\right]^{\frac{p_2}{p_2-1} \frac{2(\epsilon+1)}{\epsilon}} dx\right]^{\frac{\epsilon}{2(1+\epsilon)}},
\end{align*}
and
\begin{align*}
\mathcal{K}_2 = \left[\fint_{\Omega_{R/2}} \left|\frac{\mathbb{W}(x) - \mathbb{W}_0}{\mathbb{W}_0}\right|^{\frac{p_2}{p_2-1} \frac{2(\epsilon+1)}{\epsilon}} dx\right]^{\frac{\epsilon}{2(1+\epsilon)}}.
\end{align*}
By~\eqref{w-w0-bound} and the definition of $|\mathrm{log} \mathbb{W}|_{\mathrm{BMO}}$ in~\eqref{BMO-norm}, inequality~\eqref{well-known-2} in Lemma~\ref{lem:well-known} gives us
\begin{align}\label{Prob-5-est-4}
\mathcal{K}_1 \le C \mbox{ and } \mathcal{K}_2 \le C |\mathrm{log} \mathbb{W}|_{\mathrm{BMO}}^{\frac{p_2}{p_2-1}} \le C |\mathrm{log} \mathbb{W}|_{\mathrm{BMO}},
\end{align}
which provided $|\mathrm{log} \mathbb{W}|_{\mathrm{BMO}}$ small enough. Moreover, thanks to~\eqref{RH-v} and in the same manner as that of previous proofs, one has
\begin{align}\label{Prob-5-est-5}
\left[\fint_{\Omega_{R/2}} \left(\omega^{p_2}|\nabla \tilde{v}|^{p_2}\right)^{1+\epsilon} dx\right]^{\frac{1}{1+\epsilon}} \le \left(\fint_{\Omega_{R}} \omega^{p_2}|\nabla \tilde{v}|^{p_2} dx + 1\right).
\end{align}
By merging~\eqref{Prob-5-est-4} and~\eqref{Prob-5-est-5} into~\eqref{Prob-5-est-3}, it yields
\begin{align}\label{Prob-5-est-6}
\fint_{\Omega_{R/2}}  \omega_{0}^{p_2}  \left|\nabla \pi - \nabla\tilde{v}\right|^{p_2} dx 
 & \le \delta \fint_{\Omega_{R/2}} \omega_0^{p_2}|\nabla \tilde{v}|^{p_2} dx \notag \\
 & \qquad + C_{\delta} |\mathrm{log} \mathbb{W}|_{\mathrm{BMO}} \left(\fint_{\Omega_{R}} \omega^{p_2}|\nabla \tilde{v}|^{p_2} dx + 1\right).
\end{align}
On the other hand, let us make use of H\"older's inequality to imply that
\begin{align}\notag
\fint_{\Omega_{R/2}}  \left|\mathbb{W}(x) - \mathbb{W}_0\right|^{p_2} \left|\nabla\tilde{v}\right|^{p_2} dx & \le  C \mathcal{K}_3 \mathcal{K}_4 \left[\fint_{\Omega_{R/2}} \left(\omega^{p_2}|\nabla \tilde{v}|^{p_2}\right)^{1+\epsilon} dx\right]^{\frac{1}{1+\epsilon}},
\end{align}
where the quantities $\mathcal{K}_3, \mathcal{K}_4$ can be estimated as follows
\begin{align*}
\mathcal{K}_3 = \left[\fint_{\Omega_{R/2}} \left(\omega^{-p_2} \omega_0^{p_2}\right)^{ \frac{2(\epsilon+1)}{\epsilon}} dx\right]^{\frac{\epsilon}{2(1+\epsilon)}} \le C,
\end{align*}
and
\begin{align*}
\mathcal{K}_4 = \left[\fint_{\Omega_{R/2}} \left|\frac{\mathbb{W}(x) - \mathbb{W}_0}{\mathbb{W}_0}\right|^{p_2 \frac{2(\epsilon+1)}{\epsilon}} dx\right]^{\frac{\epsilon}{2(1+\epsilon)}} \le C |\mathrm{log} \mathbb{W}|_{\mathrm{BMO}}^{p_2} \le C |\mathrm{log} \mathbb{W}|_{\mathrm{BMO}}.
\end{align*}
It then follows from~\eqref{Prob-5-est-5} that
\begin{align}\label{Prob-5-est-7}
\fint_{\Omega_{R/2}}  \left|\mathbb{W}(x) - \mathbb{W}_0\right|^{p_2} \left|\nabla\tilde{v}\right|^{p_2} dx & \le  C|\mathrm{log} \mathbb{W}|_{\mathrm{BMO}} \left(\fint_{\Omega_{R}} \omega^{p_2}|\nabla \tilde{v}|^{p_2} dx + 1\right).
\end{align}
On the other hand, we observe that
\begin{align}
\fint_{\Omega_{R/2}}  \big|\omega |\nabla\tilde{v}| - \omega_{0} |\nabla \pi|\big|^{p_2} dx & \le C \left[ \fint_{\Omega_{R/2}}  \big|\omega |\nabla\tilde{v}| - \omega_{0} |\nabla\tilde{v}|\big|^{p_2} dx + \fint_{\Omega_{R/2}}  \big|\omega_{0} |\nabla \pi| - \omega_{0} |\nabla\tilde{v}|\big|^{p_2} dx\right] \notag \\
  & \le C \left[\fint_{\Omega_{R/2}}  \left|\mathbb{W}(x) - \mathbb{W}_0\right|^{p_2} \left|\nabla\tilde{v}\right|^{p_2} dx + \fint_{\Omega_{R/2}}  \omega_{0}^{p_2}  \left|\nabla \pi - \nabla\tilde{v}\right|^{p_2} dx \right].\notag
\end{align}
Substituting~\eqref{Prob-5-est-6} and~\eqref{Prob-5-est-7} into the above estimate, one concludes
\begin{align}\label{Prob-5-est-8}
\fint_{\Omega_{R/2}}  \big|\omega |\nabla\tilde{v}| - \omega_{0} |\nabla \pi|\big|^{p_2} dx & \le \delta \fint_{\Omega_{R/2}} \omega_0^{p_2}|\nabla \tilde{v}|^{p_2} dx \notag \\
& \qquad + C_{\delta} |\mathrm{log} \mathbb{W}|_{\mathrm{BMO}} \left(\fint_{\Omega_{R}} \omega^{p_2}|\nabla \tilde{v}|^{p_2} dx + 1\right).
\end{align}
Now, the proof of~\eqref{est-step3.2} becomes clear by combining~\eqref{est-step1b},~\eqref{Prob-4-est-1},~\eqref{Prob-5-est-0},  and~\eqref{Prob-5-est-8}. We have thus completed the proof of Lemma~\ref{lem-step3}.
\end{proof}

Now we come to the last proof of the comparison scheme, which suggests comparing $\pi$ to $\tilde{\pi}$ that has a priori $L^{\infty}$-estimate. At this stage, the final result is followed by an argument using the assumption $\mathbf{(H_3)}$. 

\begin{lemma}[The 4$^\text{th}$ step]
\label{lem-step4}
Under the assumptions of Lemma~\ref{lem-step3} and let $\pi \in \tilde{v} + W^{1,p_2}_{0}(\Omega_{R/2},\omega_0^{p_2})$ be the solution to~\eqref{Prob-5}. Then, for every $\delta \in (0,1)$, there exists $\kappa>0$ such that if $\Omega$ satisfies the $(\kappa,r_0)$-Reifenberg flatness condition,  then we are possible to find $\tilde{\pi} \in W^{1,p_2}_{0}(\Omega_{R/2},\omega_0^{p_2}) \cap L^{\infty}(\Omega_{R/4},\omega_0^{p_2})$ such that
\begin{align}\label{est-step4.1}
\sup\limits_{x \in \Omega_{R/4}} \omega_{0}^{p_2}|\nabla \tilde{\pi}|^{p_2} &\le  C  \fint_{\Omega_{2R}} \omega^{p(x)} |\nabla u|^{p(x)} dx \notag \\
& \qquad + C\left(\fint_{\Omega_{2R}} \omega^{p(x)} \big(|\nabla \phi_1|^{p(x)} +  |\nabla \phi_2|^{p(x)}\big) dx + 1\right),
\end{align}
and
\begin{align}\label{est-step4.2}
\fint_{\Omega_{R/2}} \omega_{0}^{p_2}\big|\nabla \tilde{\pi} - \nabla \pi\big|^{p_2} dx &\le \delta \fint_{\Omega_{2R}} \omega^{p(x)} |\nabla u|^{p(x)} dx \notag \\
& \qquad + C\left(\fint_{\Omega_{2R}} \omega^{p(x)} \big(|\nabla \phi_1|^{p(x)} +  |\nabla \phi_2|^{p(x)}\big) dx + 1\right).
\end{align}
\end{lemma}
\begin{proof}
The proof technique goes back to \cite [Corollary 3.1]{BY24}, from which author showed that for every $\delta \in (0,1)$, there exists $\kappa>0$ such that, if provided $(\kappa,r_0)$-Reifenberg flat condition imposed to $\Omega$, there exists a function $\tilde{\pi} \in W^{1,p_2}_{0}(\Omega_{R/2},\omega_0^{p_2}) \cap L^{\infty}(\Omega_{R/4},\omega_0^{p_2})$ such that
\begin{align}\notag 
\sup\limits_{x \in \Omega_{R/4}} \omega_{0}^{p_2}|\nabla \tilde{\pi}|^{p_2} \le  C \left( \fint_{\Omega_{R/2}} \omega_{0}^{p_2}|\nabla \pi|^{p_2} dx + 1\right),
\end{align}
and
\begin{align}\notag
\fint_{\Omega_{R/2}} \omega_{0}^{p_2}\big|\nabla \tilde{\pi} - \nabla \pi\big|^{p_2} dx \le \delta \left( \fint_{\Omega_{R/2}} \omega_{0}^{p_2}|\nabla \pi|^{p_2} dx + 1\right).
\end{align}
Gathering all the comparison estimates from Lemma~\ref{lem-step1},~\ref{lem-step2} and~\ref{lem-step3}, one can easily obtain~\eqref{est-step4.1} and~\eqref{est-step4.2}. This finally finishes the proof. 
\end{proof}

\section{Estimates on level sets}
\label{sec:levelset}

In this section, we turn our attention to proving the level set estimates, which are based on the comparison estimates. The step of constructing the level-set inequality is a basic and important step when dealing with regularity and qualitative properties of solutions in the described approach. The main idea of the proof stemmed from a basic technique relying on a Calder\'on-Zygmund-type covering lemma applied to suitable level sets (see~\cite{AM2007, CP1998, BW2004}). In the context of solutions to problem~\eqref{OP-var} involving degenerate matrix-weights with obstacles, one further needs some delicate arguments to adapt the idea to our setting. Let us describe in detail the two significant improvements in the present paper. The first one, as shown in Section~\ref{sec:comparison}, the reader will see in action the tricks designed for the comparison scheme. To be more precise, due to the complicated structure of the models, the comparison strategy cannot be performed directly from the original weak solution $u$ to~\eqref{OP-var} to $\tilde{\pi}$, but rather through several steps, which is difficult to handle. The next one, as an advance of our contribution, will be highlighted in this section. Here, in our argument, we specifically establish an \emph{optimal} level-set inequality, which reduces the dependence of auxiliary parameters on the running $\varepsilon$ in measurable level sets. This improvement guarantees that the regularity estimates in certain function spaces can be derived directly in the subsequent step; meanwhile, the scaling parameters of such spaces are no longer constrained by adding extra assumptions. This is one of the key tools for our main results, and we also believe that the so-called large-scale level-set estimates could be adapted to treat gradient regularity results, even in a vast array of more generalized function spaces. 

\begin{theorem}
\label{theo-LV}
Let $u \in \mathbb{K}^\omega_{\phi_1,\phi_2}$ be a weak solution to the main problem~\eqref{OP-var} under assumptions~\eqref{Pi-2}-\eqref{def-phi-12}, and with given $\beta \in [0,n)$, $\varepsilon \in (0,1)$. Then, one can find some constants $\sigma \in (0,3^{-n})$, $\kappa = \kappa(\varepsilon) \in (0,\sigma)$ and $\lambda_0 = \lambda_0(\kappa)$ such that if  $\mathbf{(H_1)}$-$\mathbf{(H_3)}$ satisfy for some $r_0>0$, then the following level-set inequality holds true:
\begin{align}\label{ineq-LV}
\left|\left\{\mathbf{M}_{\beta}\big(\omega^{p(\cdot)} |\nabla u|^{p(\cdot)}\big) > \lambda, \, \mathbf{M}_{\beta}\mathbb{F}_{\omega} \le \kappa \lambda\right\}\right| \le C \varepsilon \left|\left\{\mathbf{M}_{\beta}\big(\omega^{p(\cdot)} |\nabla u|^{p(\cdot)}\big) > \sigma\lambda\right\}\right|,
\end{align}
for all $\lambda \ge \lambda_0$. Here, the constant $C$ depends on $\textsc{dataset}$.
\end{theorem}
\begin{proof} 
First, let us fix $\beta \in [0,n)$ and $\varepsilon \in (0,1)$. For the sake of readability, we will denote two measurable subsets of $\Omega$ on the left-hand side and right-hand side of~\eqref{ineq-LV} as $V_{1,\varepsilon}^{\lambda}$ and $V_{2}^{\lambda}$, respectively. More precisely, for any $\lambda>0$, one often writes
\begin{align}\notag 
 V_{1,\varepsilon}^{\lambda} := &\left\{\mathbf{M}_{\beta}\big(\omega^{p(\cdot)} |\nabla u|^{p(\cdot)}\big) > \lambda, \, \mathbf{M}_{\beta}\mathbb{F}_{\omega} \le \kappa \lambda\right\}, \quad V_{2}^{\lambda} := \left\{\mathbf{M}_{\beta}\big(\omega^{p(\cdot)} |\nabla u|^{p(\cdot)}\big) > \sigma \lambda\right\}.
\end{align}
It is worth noticing that due to the dependence of $\kappa$ on $\varepsilon$, the first level set also depends on $\varepsilon$, meanwhile, the remaining level set does not.  \\
It is remarkable that the conclusion obviously occurs for any empty subset $V_{1,\varepsilon}^{\lambda}$. Therefore, we are allowed to consider some $y \in V_{1,\varepsilon}^{\lambda}$, which implies $\mathbf{M}_{\beta}\mathbb{F}_{\omega}(y) \le \kappa \lambda$. By choosing $T_0 := 2\mathrm{diam}(\Omega)$ and combining to the global estimate~\eqref{glob-est}, it obtains
\begin{align}\label{def-V1-ineq}
\int_{\Omega} \omega^{p(x)} |\nabla u|^{p(x)} dx  \leq C\int_{B_{T_0}(y)} \mathbb{F}_{\omega}(x)dx  \leq  C|B_{T_0}(y)|  T_0^{-\beta}  \mathbf{M}_{\beta}\mathbb{F}_{\omega}(y) \le C T_0^{n-\beta}\kappa \lambda.
\end{align}
Regarding the fractional maximal operator $\mathbf{M}_{\beta}$, from Lemma~\ref{lem:bound-M-beta}, it satisfies a boundedness property. More precisely, there exists a constant $C=C(n,\beta)>0$ such that
\begin{align}\label{bound-M-al}
\sup_{\lambda>0} \lambda\left|\left\{x \in \mathbb{R}^n: \ \mathbf{M}_{\beta}{f}(x)>\lambda\right\}\right|^{1-\frac{\beta}{n}} \le C \|{f}\|_{L^1(\mathbb{R}^n)}, 
\end{align}
for all ${f} \in L^1(\mathbb{R}^n)$. At this stage, taking~\eqref{def-V1-ineq} into account and applying~\eqref{bound-M-al} with $f = \chi_{\Omega}\omega^{p(\cdot)} |\nabla u|^{p(\cdot)}$, it follows that
\begin{align}
\left|V_{1,\varepsilon}^{\lambda}\right| & \leq \left|\left\{\mathbf{M}_{\beta}\big(\omega^{p(\cdot)} |\nabla u|^{p(\cdot)}\big) > \lambda\right\}\right|  \leq \left(\frac{C}{\lambda}\int_{\Omega} \omega^{p(x)} |\nabla u|^{p(x)} dx  \right)^{\frac{n}{n-\beta}}  \leq CT_0^{n}\kappa^{\frac{n}{n-\beta}}  \le \varepsilon R_0^n, \label{step0}
\end{align}
from which, the last inequality in~\eqref{step0} holds by fixing a $R_0 \in (0,r_0]$ and $\kappa$ satisfying
\begin{align}\label{kappa-1}
0<\kappa \leq (C^{-1}\varepsilon)^{1-\frac{\beta}{n}} \left( R_0T_0^{-1}\right)^{n - \beta}. 
\end{align}

In a next step, for every $x_0\in \Omega$ and $0<\rho \le R_0$, we prove that 
\begin{align}\notag
\Omega_{\rho}(x_0) \not\subset V_{2}^{\lambda} \Longrightarrow \left|B_{\rho}(x_0)\cap V_{1,\varepsilon}^{\lambda}\right| < \varepsilon |B_{\rho}(x_0)|.
\end{align}
Being completely analogous to the preceding argument, we can assume that $B_{\rho}(x_0)\cap V_{1,\varepsilon}^{\lambda} \neq \emptyset$  and $\Omega_{\rho}(x_0) \not\subset V_{2}^{\lambda}$, which verifies the existence of some $x_1 \in B_{\rho}(x_0)$  and $x_2 \in \Omega_{\rho}(x_0)$ such that 
\begin{align}\label{def-x-1}
\sup_{r>0} r^{\beta}\fint_{B_r(x_1)} \mathbb{F}_{\omega}(x) dx \leq \kappa \lambda, \ \mbox{ and } \   \sup_{r>0} r^{\beta} \fint_{B_r(x_2)} \omega^{p(x)} |\nabla u|^{p(x)} dx \leq \sigma \lambda.
\end{align}
Having this at hand, one further needs to prove that
\begin{align}\label{lem-LV-ineq-1}
\left|B_{\rho}(x_0)\cap V_{1,\varepsilon}^{\lambda}\right| < \varepsilon |B_{\rho}(x_0)|.
\end{align}
For any $y \in B_{\rho}(x_0)\cap V_{1,\varepsilon}^{\lambda}$, we are able to decompose 
$$\mathbf{M}_{\beta}\big(\omega^{p(\cdot)} |\nabla u|^{p(\cdot)}\big)(y) = \max\left\lbrace \mathbf{M}^{\rho}_{\beta}\big(\omega^{p(\cdot)} |\nabla u|^{p(\cdot)}\big)(y),  \mathbf{T}^{\rho}_{\beta}\big(\omega^{p(\cdot)} |\nabla u|^{p(\cdot)}\big)(y)\right\rbrace,$$
where two cut-off maximal functions $\mathbf{M}^{\rho}$ and $\mathbf{T}^{\rho}$ are described in Definition~\ref{def:M_beta}. By~\eqref{def-x-1} and $B_r(y) \subset B_{3r}(x_2)$ for $r>\rho$, it ensures that
\begin{align*}
\mathbf{T}^{\rho}_{\beta}\big(\omega^{p(\cdot)} |\nabla u|^{p(\cdot)}\big)(y) & \leq 3^n \sup_{r > \rho} r^{\beta} \fint_{B_{3r}(x_2)} \omega^{p(x)} |\nabla u|^{p(x)} dx  \leq 3^n\sigma\lambda.
\end{align*}
Let us choose $\sigma \leq 3^{-n}$, which leads to $\mathbf{T}^{\rho}_{\beta}\big(\omega^{p(\cdot)} |\nabla u|^{p(\cdot)}\big)(y) \leq \lambda$, for all $y \in B_{\rho}(x_0)\cap V_{1,\varepsilon}^{\lambda}$. It allows us to conclude
\begin{align*}
B_{\rho}(x_0)\cap V_{1,\varepsilon}^{\lambda}  \subset B_{\rho}(x_0)\cap \left\{\mathbf{M}_{\beta}\big(\omega^{p(\cdot)} |\nabla u|^{p(\cdot)}\big) > \lambda\right\} = B_{\rho}(x_0)\cap \left\{\mathbf{M}_{\beta}^{\rho}\big(\omega^{p(\cdot)} |\nabla u|^{p(\cdot)}\big) > \lambda\right\}.
\end{align*}
One can readily verify that 
$$\mathbf{M}_{\beta}^{\rho}\big(\omega^{p(\cdot)} |\nabla u|^{p(\cdot)}\big) = \mathbf{M}_{\beta}^{\rho}\big(\chi_{B_{2\rho}(x_0)} \omega^{p(\cdot)} |\nabla u|^{p(\cdot)}\big)$$ 
in the ball $B_{\rho}(x_0)$. Consequently, we can conclude that
\begin{align}\label{max-oper-bounded}
B_{\rho}(x_0)\cap V_{1,\varepsilon}^{\lambda}  \subset B_{\rho}(x_0)\cap \left\{\mathbf{M}_{\beta}^{\rho}\big(\chi_{B_{2\rho}(x_0)} \omega^{p(\cdot)} |\nabla u|^{p(\cdot)}\big) > \lambda\right\}.
\end{align} 
At this step, we will use the comparison estimates presented in the previous section. Thanks to Lemma~\ref{lem-step1}, one can find $v \in W^{1,p(\cdot)}\left(\Omega_{16\rho}(x_0),\omega^{p(\cdot)}\right) \cap W^{1,p_2}\left(\Omega_{8\rho}(x_0),\omega^{p_2}\right)$ such that
\begin{align}\label{STEP-1}
\fint_{\Omega_{16\rho}(x_0)} \omega^{p(x)}  |\nabla u - \nabla v|^{p(x)} & dx \le \delta \fint_{\Omega_{16\rho}(x_0)} \omega^{p(x)}|\nabla u|^{p(x)} dx \notag \\
& + C_{\delta} \fint_{\Omega_{16\rho}(x_0)} \omega^{p(x)} \big(|\mathbf{F}|^{p(x)} + |\nabla \phi_1|^{p(x)} + |\nabla \phi_2|^{p(x)}\big) dx,
\end{align}
for every $\delta \in (0,1)$. Lemma~\ref{lem-step2} ensures the existence of $\tilde{v} \in v + W^{1,p_2}_{0}(\Omega_{8\rho}(x_0),\omega^{p_2})$ such that
\begin{align}\label{STEP-2}
\fint_{\Omega_{8\rho}(x_0)} & \omega^{p_2}|\nabla v - \nabla \tilde{v}|^{p_2} dx \le \left(\delta + C_{\delta} \kappa^{\frac{p_2}{p_2-1}}\right) \fint_{\Omega_{16\rho}(x_0)} \omega^{p(x)}|\nabla u|^{p(x)} dx \notag \\
& \qquad +   C_{\delta} \left[1 + \kappa^{\frac{p_2}{p_2-1}} \fint_{\Omega_{16\rho}(x_0)} \omega^{p(x)} \big(|\nabla \phi_1|^{p(x)} + |\nabla \phi_2|^{p(x)}\big) dx \right],
\end{align}
where $p_2 = \sup\limits_{x \in \Omega_{8\rho}(x_0)} p(x)$. Next, Lemma~\ref{lem-step3} gives us a function $\pi \in \tilde{v} + W^{1,p_2}_{0}(\Omega_{4\rho}(x_0),\omega_0^{p_2})$ such that
\begin{align}\label{STEP-3}
\fint_{\Omega_{4\rho}(x_0)}  &\big|\omega |\nabla\tilde{v}| - \omega_{0} |\nabla \pi|\big|^{p_2} dx  \le \left(\delta + C_{\delta} |\mathrm{log} \mathbb{W}|_{\mathrm{BMO}}\right) \fint_{\Omega_{16\rho}(x_0)} \omega^{p(x)}|\nabla u|^{p(x)} dx \notag \\
& \qquad +   C_{\delta} |\mathrm{log} \mathbb{W}|_{\mathrm{BMO}} \left[1 +  \fint_{\Omega_{16\rho}(x_0)} \omega^{p(x)} \big(|\nabla \phi_1|^{p(x)} + |\nabla \phi_2|^{p(x)}\big) dx \right],
\end{align}
provided $|\mathrm{log} \mathbb{W}|_{\mathrm{BMO}}$ small enough, where 
$$\mathbb{W}_0 = \langle \mathbb{W} \rangle_{B_{4\rho}(x_0)}^{\mathrm{log}}, \mbox{ and } \omega_0 = \langle \omega \rangle_{B_{4\rho}(x_0)}^{\mathrm{log}}.$$
According Lemma~\ref{lem-step4}, if given $\Omega$ satisfies the $(\kappa,r_0)$-Reifenberg flat condition, then there exists $\tilde{\pi} \in W^{1,p_2}_{0}(\Omega_{4\rho}(x_0),\omega_0^{p_2}) \cap L^{\infty}(\Omega_{2\rho}(x_0),\omega_0^{p_2})$ such that 
\begin{align}\label{STEP-41}
\sup\limits_{x \in \Omega_{2\rho}(x_0)} \omega_{0}^{p_2}|\nabla \tilde{\pi}|^{p_2} &\le  C  \fint_{\Omega_{16\rho}(x_0)} \omega^{p(x)} |\nabla u|^{p(x)} dx \notag \\
& \qquad + C\left(\fint_{\Omega_{16\rho}(x_0)} \omega^{p(x)} \big(|\nabla \phi_1|^{p(x)} +  |\nabla \phi_2|^{p(x)}\big) dx + 1\right),
\end{align}
and
\begin{align}\label{STEP-4}
\fint_{\Omega_{4\rho}(x_0)} \omega_{0}^{p_2}\big|\nabla \tilde{\pi} - \nabla \pi\big|^{p_2} dx & \le \delta  \fint_{\Omega_{16\rho}(x_0)} \omega^{p(x)} |\nabla u|^{p(x)} dx \notag \\
& \qquad + C\left(\fint_{\Omega_{16\rho}(x_0)} \omega^{p(x)} \big(|\nabla \phi_1|^{p(x)} +  |\nabla \phi_2|^{p(x)}\big) dx + 1\right).
\end{align}
On the other hand, for every $x \in B_{\rho}(x_0)$, it is easily to see that 
\begin{align*}
|\zeta|^{p(x)} \le 2^{p_{\mathrm{max}}-1}\left(|\zeta-\eta|^{p(x)} + |\eta|^{p(x)}\right)
\mbox{ and }
|\eta|^{p(x)} \le 2^{p_{\mathrm{max}}-1} \left(|\eta|^{p_2} + 1\right),
\end{align*}
for all $\zeta, \eta \in \mathbb{R}^n$. Taking the advantages of these estimates, we deduce that
\begin{align*}
\omega^{p(x)}|\nabla u|^{p(x)} & \leq C \left(  \omega^{p(x)}\left| \nabla u - \nabla v \right|^{p(x)} + \omega^{p_2} \left|\nabla v \right|^{p_2} + 1\right),
\end{align*}
and
\begin{align*}
\omega^{p_2} \left|\nabla v \right|^{p_2} & \le C \Big(\omega^{p_2}\left| \nabla v - \nabla \tilde{v} \right|^{p_2}  + \big| \omega|\nabla \tilde{v}| - \omega_{0}|\nabla \pi| \big|^{p_2}  +  \omega_{0}^{p_2} \big| \nabla \pi  - \nabla \tilde{\pi} \big|^{p_2}  +  \omega_{0}^{p_2} \big| \nabla \tilde{\pi} \big|^{p_2}\Big).
\end{align*}
Thus, we are able to conclude that 
\begin{align*}
\mathbf{M}^{\rho}_{\beta} \big( \chi_{B_{2\rho}(x_0)} \omega^{p(\cdot)}|\nabla u|^{p(\cdot)} \big) &\leq \ \mathbf{M}^{\rho}_{\beta} \big(C_1 \chi_{B_{2\rho}(x_0)} \omega^{p(\cdot)}|\nabla u - \nabla v|^{p(\cdot)} \big)\\ 
& \ \quad + \mathbf{M}^{\rho}_{\beta} \big(C_1 \chi_{B_{2\rho}(x_0)} \omega^{p_2}\left| \nabla v - \nabla \tilde{v} \right|^{p_2} \big) \\
& \ \quad \quad + \mathbf{M}^{\rho}_{\beta} \big(C_1 \chi_{B_{2\rho}(x_0)} \big|\omega|\nabla \tilde{v}| - \omega_{0}|\nabla \pi|\big|^{p_2} \big) \\
& \ \quad \quad \quad + \mathbf{M}^{\rho}_{\beta} \big(C_1 \chi_{B_{2\rho}(x_0)} \omega_{0}^{p_2}|\nabla \pi  - \nabla \tilde{\pi}|^{p_2} \big) \\
& \ \quad \quad \quad \quad + \mathbf{M}^{\rho}_{\beta} \big(C_1 \chi_{B_{2\rho}(x_0)} \big( \omega_{0}^{p_2}|\nabla \tilde{\pi}|^{p_2} + 1\big) \big), \ \mbox{ in } B_{\rho}(x_0).
\end{align*}
Combining with~\eqref{max-oper-bounded}, it yields
\begin{align}\notag
|B_{\rho}(x_0)\cap V_{1,\varepsilon}^{\lambda}|  &\le \left|B_{\rho}(x_0)\cap \left\{\mathbf{M}_{\beta}^{\rho}\big(\chi_{B_{2\rho}(x_0)} \omega^{p(\cdot)} |\nabla u|^{p(\cdot)}\big) > \lambda\right\}\right| \\
& \le \left|B_{\rho}(x_0)\cap \left\{\mathbf{M}_{\beta}^{\rho}\big(C_1\chi_{B_{2\rho}(x_0)} \omega^{p(\cdot)} |\nabla u - \nabla v|^{p(\cdot)}\big) > \frac{\lambda}{5}\right\}\right| \notag \\
& \hspace{.5cm} + \left|B_{\rho}(x_0)\cap \left\{\mathbf{M}_{\beta}^{\rho}\big(C_1 \chi_{B_{2\rho}(x_0)} \omega^{p_2}\left| \nabla v - \nabla \tilde{v} \right|^{p_2} \big) > \frac{\lambda}{5}\right\}\right| \notag \\
& \hspace{1cm}  + \left|B_{\rho}(x_0)\cap \left\{\mathbf{M}_{\beta}^{\rho}\big(C_1 \chi_{B_{2\rho}(x_0)} \big|\omega|\nabla \tilde{v}| - \omega_{0}|\nabla \pi|\big|^{p_2} \big) > \frac{\lambda}{5}\right\}\right| \notag \\
& \hspace{1.5cm} + \left|B_{\rho}(x_0)\cap \left\{\mathbf{M}_{\beta}^{\rho}\big(C_1 \chi_{B_{2\rho}(x_0)} \omega_{0}^{p_2}|\nabla \pi  - \nabla \tilde{\pi}|^{p_2} \big) > \frac{\lambda}{5}\right\}\right| \notag \\
& \hspace{2cm} + \left|B_{\rho}(x_0)\cap \left\{\mathbf{M}_{\beta}^{\rho}\big(C_1 \chi_{B_{2\rho}(x_0)} \big( \omega_{0}^{p_2}|\nabla \tilde{\pi}|^{p_2} + 1\big) \big) > \frac{\lambda}{5}\right\}\right| \notag \\
& \le \mathcal{N}_1 + \mathcal{N}_2 + \mathcal{N}_3 + \mathcal{N}_4 + \mathcal{N}_5,  \label{N-sum}
\end{align} 
where we define
\begin{align*}
\mathcal{N}_1&:= \left|B_{\rho}(x_0)\cap \left\{\mathbf{M}_{\beta}^{\rho}\big(C_1\chi_{B_{2\rho}(x_0)} \omega^{p(\cdot)} |\nabla u - \nabla v|^{p(\cdot)}\big) > \frac{\lambda}{5}\right\}\right|; \\
\mathcal{N}_2&:=\left|B_{\rho}(x_0)\cap \left\{\mathbf{M}_{\beta}^{\rho}\big(C_1 \chi_{B_{2\rho}(x_0)} \omega^{p_2}\left| \nabla v - \nabla \tilde{v} \right|^{p_2} \big) > \frac{\lambda}{5}\right\}\right|; \\
\mathcal{N}_3&:=\left|B_{\rho}(x_0)\cap \left\{\mathbf{M}_{\beta}^{\rho}\big(C_1 \chi_{B_{2\rho}(x_0)} \big|\omega|\nabla \tilde{v}| - \omega_{0}|\nabla \pi|\big|^{p_2} \big) > \frac{\lambda}{5}\right\}\right|; \\
\mathcal{N}_4&:=\left|B_{\rho}(x_0)\cap \left\{\mathbf{M}_{\beta}^{\rho}\big(C_1 \chi_{B_{2\rho}(x_0)} \omega_{0}^{p_2}|\nabla \pi  - \nabla \tilde{\pi}|^{p_2} \big) > \frac{\lambda}{5}\right\}\right|; \\
\mathcal{N}_5&:=\left|B_{\rho}(x_0)\cap \left\{\mathbf{M}_{\beta}^{\rho}\big(C_1 \chi_{B_{2\rho}(x_0)} \big( \omega_{0}^{p_2}|\nabla \tilde{\pi}|^{p_2} + 1\big) \big) > \frac{\lambda}{5}\right\}\right|. 
\end{align*}
To estimate the first four terms $\mathcal{N}_i$ for $i=1,2,3,4$, let us first use the boundedness property of the maximal operator $\mathbf{M}_{\beta}$ in~\eqref{bound-M-al}, it  yields that
\begin{align}
\sum_{j=1}^{4}\mathcal{N}_j & \le C \left(\lambda^{-1} \int_{\Omega_{2\rho}(x_0)} \omega^{p(x)} |\nabla u - \nabla v|^{p(x)} dx \right)^{\frac{n}{n-\beta}} \notag \\
& \hspace{0.5cm} + C \left(\lambda^{-1} \int_{\Omega_{2\rho}(x_0)} \omega^{p_2}\left| \nabla v - \nabla \tilde{v} \right|^{p_2} dx \right)^{\frac{n}{n-\beta}} \notag \\
& \hspace{1cm} + C \left(\lambda^{-1} \int_{\Omega_{2\rho}(x_0)} \big|\omega|\nabla \tilde{v}| - \omega_{0}|\nabla \pi|\big|^{p_2} dx \right)^{\frac{n}{n-\beta}} \notag \\
& \hspace{1.5cm} + C \left(\lambda^{-1} \int_{\Omega_{2\rho}(x_0)} \omega_{0}^{p_2}|\nabla \pi  - \nabla \tilde{\pi}|^{p_2} dx \right)^{\frac{n}{n-\beta}}. \label{N-sum-14}
\end{align}
Here, note that the first integral term of~\eqref{N-sum-14} can be evaluated by combining~\eqref{STEP-1} and~\eqref{def-x-1}. Indeed, since $\Omega_{16\rho}(x_0) \subset \Omega_{17\rho}(x_2)$, from~\eqref{def-x-1}, it concludes that
\begin{align}\notag
\fint_{\Omega_{16\rho}(x_0)} \omega^{p(x)}  |\nabla u|^{p(x)} dx \le C \fint_{\Omega_{17\rho}(x_2)} \omega^{p(x)}  |\nabla u|^{p(x)} dx \le C\rho^{-\beta} \sigma \lambda.
\end{align}
The argument is somewhat similar that leads to
\begin{align}\notag
\fint_{\Omega_{16\rho}(x_0)} \omega^{p(x)} \big(|\mathbf{F}|^{p(x)} + |\nabla \phi_1|^{p(x)} + |\nabla \phi_2|^{p(x)}\big) dx \le C\rho^{-\beta} \kappa \lambda.
\end{align}
Combining these above estimates with~\eqref{STEP-1}, one obtains 
\begin{align}\notag
\int_{\Omega_{2\rho}(x_0)} \omega^{p(x)}  |\nabla u - \nabla v|^{p(x)} dx \le C\int_{\Omega_{16\rho}(x_0)} \omega^{p(x)}  |\nabla u - \nabla v|^{p(x)} dx \le C \rho^{n-\beta} \lambda \left(\delta \sigma + C_{\delta} \kappa\right),
\end{align}
which directly implies
\begin{align}\label{STEP-1B}
\left(\lambda^{-1}\int_{\Omega_{2\rho}(x_0)} \omega^{p(x)}  |\nabla u - \nabla v|^{p(x)}  dx\right)^{\frac{n}{n-\beta}} \le C \rho^{n} \big(\delta \sigma + C_{\delta} \kappa\big)^{\frac{n}{n-\beta}},
\end{align}
Again, the siminar argument allows us to infer from~\eqref{STEP-2} that
\begin{align}\notag
\int_{\Omega_{2\rho}(x_0)} \omega^{p_2}\left| \nabla v - \nabla \tilde{v} \right|^{p_2} dx  \le C  \rho^n\left[\left(\delta + C_{\delta} \kappa^{\frac{p_2}{p_2-1}}\right) \rho^{-\beta}\sigma \lambda + C_{\delta} \left(1 + \kappa^{\frac{p_2}{p_2-1}} \rho^{-\beta} \kappa \lambda\right)\right].
\end{align}
For the term appearing on the right-hand side of this estimate,  we restrict $\lambda$ such that
\begin{align}\label{lamb-0}
\lambda \ge \lambda_0: = \kappa^{-1}r_0^{\beta} \ge \kappa^{-1}\rho^{\beta},
\end{align}
which guarantees that $1 \le \rho^{-\beta}\kappa \lambda$. Therefore, it enables us to obtain
\begin{align}\label{STEP-2B}
\left(\lambda^{-1} \int_{\Omega_{2\rho}(x_0)} \omega^{p_2}\left| \nabla v - \nabla \tilde{v} \right|^{p_2} dx \right)^{\frac{n}{n-\beta}} \le C \rho^{n} \left[\left(\delta + C_{\delta} \kappa^{\frac{p_2}{p_2-1}}\right) \sigma + C_{\delta} \left(\kappa + \kappa^{\frac{p_2}{p_2-1}} \kappa\right)\right]^{\frac{n}{n-\beta}}.
\end{align}
Next, the remaining two terms on the right-hand side of~\eqref{N-sum} are treated via the comparison estimates presented in~\eqref{STEP-3} and~\eqref{STEP-4}. In such a way, making use of conditions~\eqref{lamb-0} and $|\mathrm{log} \mathbb{W}|_{\mathrm{BMO}} \le \kappa$, it yields
\begin{align}\label{STEP-3B}
\left(\lambda^{-1} \int_{\Omega_{2\rho}(x_0)} \big|\omega|\nabla \tilde{v}| - \omega_{0}|\nabla \pi|\big|^{p_2} dx \right)^{\frac{n}{n-\beta}} &\le C \rho^{n} \big[\left(\delta + C_{\delta} |\mathrm{log} \mathbb{W}|_{\mathrm{BMO}}\right) \sigma + C_{\delta} \kappa |\mathrm{log} \mathbb{W}|_{\mathrm{BMO}}\big]^{\frac{n}{n-\beta}} \notag \\
& \le C \rho^{n} \big[\left(\delta + C_{\delta} \kappa\right) \sigma + C_{\delta} \kappa^2\big]^{\frac{n}{n-\beta}},
\end{align}
and
\begin{align}\label{STEP-4B}
\left(\lambda^{-1} \int_{\Omega_{2\rho}(x_0)} \omega_{0}^{p_2}|\nabla \pi  - \nabla \tilde{\pi}|^{p_2} dx \right)^{\frac{n}{n-\beta}} \le C \rho^{n} \left(\delta\sigma + \kappa\right)^{\frac{n}{n-\beta}}.
\end{align}
Gathering all estimates~\eqref{STEP-1B},~\eqref{STEP-2B},~\eqref{STEP-3B},~\eqref{STEP-4B}, and then reabsorbing them into~\eqref{N-sum-14}, we arrive at
\begin{align}\notag
\sum_{j=1}^{4}\mathcal{N}_j & \le C \rho^n \left[\delta \sigma + C_{\delta} \kappa + \big(\delta + C_{\delta} \kappa^{\frac{p_2}{p_2-1}}\big) \sigma + C_{\delta} \kappa \big(1 + \kappa^{\frac{p_2}{p_2-1}}\big) + \left(\delta + C_{\delta} \kappa\right) \sigma + C_{\delta} \kappa^2 + \delta\sigma + \kappa\right]^{\frac{n}{n-\beta}}.
\end{align}
Keep in mind that $\sigma,\kappa \in (0,1)$, it allows us to conclude
\begin{align}\label{N-14-final}
\sum_{j=1}^{4}\mathcal{N}_j & \le C \rho^n \big(\delta  + C_{\delta} \kappa\big)^{\frac{n}{n-\beta}} \le C_2 \big(\delta  + C_{\delta} \kappa\big)^{\frac{n}{n-\beta}} |B_{\rho}(x_0)|.
\end{align}
At this point, we turn our attention to estimating the last term, $\mathcal{N}_5$. It is worth noting that this term will vanish for $\sigma>0$ sufficiently small and $\lambda>0$ large. Indeed, for every $y \in B_{\rho}(x_0)$, from~\eqref{STEP-41}, one has
\begin{align}\notag 
\mathbf{M}^{\rho}_{\beta}\big(C \chi_{B_{2\rho}(x_0)} \big( \omega_{0}^{p_2}|\nabla \tilde{\pi}|^{p_2} + 1\big) \big)(y) & = C\sup_{r \leq \rho} r^{\beta} \fint_{B_{2\rho}(x_0)}  \big( \omega_{0}^{p_2}|\nabla \tilde{\pi}|^{p_2} + 1\big)  dx \notag \\
& \le C \rho^{\beta}  \fint_{\Omega_{16\rho}(x_0)} \omega^{p(x)} |\nabla u|^{p(x)} dx \notag \\
& \qquad + C \rho^{\beta} \left(\fint_{\Omega_{16\rho}(x_0)} \omega^{p(x)} \big(|\nabla \phi_1|^{p(x)} +  |\nabla \phi_2|^{p(x)}\big) dx + 1\right) \notag \\
& \le C_3 \left(\kappa\lambda + \sigma\lambda + \rho^{\beta}\right). \label{N-5}
\end{align}
Therefore, for all $\lambda \ge \lambda_0$ given in~\eqref{lamb-0}, we may choose suitable $\kappa$ and $\sigma$ satisfying
$$\kappa \le \sigma \le \min\left\{\frac{1}{3^n}, \frac{1}{15C_3}\right\}.$$ 
Consequently, making use of~\eqref{N-5}, we obtain that 
\begin{align}\notag
\mathbf{M}^{\rho}_{\beta}\big(C \chi_{B_{2\rho}(x_0)} \big( \omega_{0}^{p_2}|\nabla \tilde{\pi}|^{p_2} + 1\big) \big)(y) \le \frac{\lambda}{5}.
\end{align}
In other words, one can conclude that $\mathcal{N}_5 = 0$ for all $\lambda \ge \lambda_0$ and $\sigma$ small enough. We next insert this estimate into~\eqref{N-sum} and~\eqref{N-14-final} to get
\begin{align}\label{Final-est}
|B_{\rho}(x_0)\cap V_{1,\varepsilon}^{\lambda}|  &\le C_2 \big(\delta  + C_{\delta} \kappa\big)^{\frac{n}{n-\beta}} |B_{\rho}(x_0)|.
\end{align}
At this step, it is possible to choose $\delta$ and $\kappa$ sufficiently small in~\eqref{Final-est} such that
\begin{align}\label{kappa-2}
C_2 \big(\delta  + C_{\delta} \kappa\big)^{\frac{n}{n-\beta}} < \varepsilon,
\end{align}
which allows us to get~\eqref{lem-LV-ineq-1} from~\eqref{Final-est}. With the constraints that $\kappa$ satisfies~\eqref{kappa-1} and~\eqref{kappa-2}, $\lambda \ge \lambda_0$, and $\sigma$ is a given small constant at hand, we are going to prove the following two statements
\begin{itemize}
\item[$(i)$] $\left|V_{1,\varepsilon}^{\lambda}\right| \le \varepsilon R_0^n$; and
\item[$(ii)$] for every $x_0\in \Omega$ and $\rho \in (0, R_0]$, we claim that $\Omega_{\rho}(x_0) \not\subset V_{2}^{\lambda} \Longrightarrow \left|B_{\rho}(x_0)\cap V_{1,\varepsilon}^{\lambda}\right| < \varepsilon |B_{\rho}(x_0)|$.
\end{itemize}
By the reasoning from above and in virtue of Lemma~\ref{lem:Vitali}, we conclude~\eqref{ineq-LV}, and we have finally finished the proof of Theorem~\ref{theo-LV}.
\end{proof}

\section{Statements and proofs of main results}
\label{sec:main}

In this section, we provide the formal statements of the main results with their full generality, via Theorem~\ref{theo-Ge1} and~\ref{theo-Ge2}. With all the preceding results at hand, the proofs of the main theorems are also accomplished in two parts. Here, the level-set estimate proved in Theorem~\ref{theo-LV} is employed to conclude gradient regularity for solutions to the two-ostacle-problem~\ref{OP-var} to a larger class of functional settings. For the sake of exposition, to begin with, we shall restrict ourselves to analyzing results in Lebesgue spaces, as the most fundamental example of function spaces.  

We now state the fundamental assumptions underlying the two main theorems that follow. We suppose that $\mathbb{W}: \Omega \to \mathbb{R}^{n\times n}_{\mathrm{sym}^+}$ is a matrix weight satisfying~\eqref{Pi-2} and $p(\cdot)$ is a continuous function satisfying~\eqref{cond:pq1}. Let $\mathbf{F} \in L^{p(\cdot)}\big(\Omega,\omega^{p(\cdot)}\big)$ be a given data and $\phi_1, \phi_2 \in W^{1,p(\cdot)}\big(\Omega,\omega^{p(\cdot)}\big)$ be two given obstacles as in~\eqref{def-phi-12}. Assume further that $u \in \mathbb{K}_{\phi_1,\phi_2}^\omega$ is a solution to~\eqref{OP-var} and $\mathbb{F}_{\omega}$ is defined as in~\eqref{def-UF}.

\begin{theorem}
\label{theo-Ge1}
For every $\beta \in [0,n)$ and $\gamma>1$, there exists a constant $\kappa \in (0,1)$ such that if the assumptions $\mathbf{(H_1)}$-$\mathbf{(H_3)}$ satisfy for some $r_0>0$, then the following fractional gradient estimate holds
\begin{align}\label{ineq-Ge1}
\left\|\mathbf{M}_{\beta}\big(\omega^{p(\cdot)} |\nabla u|^{p(\cdot)}\big)\right\|_{L^{\gamma}(\Omega)} & \le C \left[ 1 +  \left\|\mathbf{M}_{\beta}\mathbb{F}_{\omega}\right\|_{L^{\gamma}(\Omega)}\right].
\end{align}
Here, the constant $C$ depends on $\textsc{dataset}$ and $\gamma$.
\end{theorem}
\begin{proof}
The proof of the theorem is closely intertwined with the conclusion of Theorem~\ref{theo-LV}, which provides us with the key ingredient in the argument. In particular, for every $\varepsilon \in (0,1)$, it claims the existence of $\sigma \in (0,3^{-n})$ and $\kappa_{\varepsilon} \in (0,\sigma)$ such that if $\mathbf{(H_1)}$-$\mathbf{(H_3)}$ hold for $\kappa_{\varepsilon}$ and for some $r_0>0$, the following level-set inequality holds true
\begin{align}\label{Ge1-est1}
\left|\left\{\mathbf{M}_{\beta}\big(\omega^{p(\cdot)} |\nabla u|^{p(\cdot)}\big) > \lambda, \, \kappa_{\varepsilon}^{-1}\mathbf{M}_{\beta}\mathbb{F}_{\omega} \le  \lambda\right\}\right| \le C \varepsilon \left|\left\{\sigma^{-1}\mathbf{M}_{\beta}\big(\omega^{p(\cdot)} |\nabla u|^{p(\cdot)}\big) > \lambda\right\}\right|,
\end{align}
for all $\lambda \ge \lambda_0 := \kappa_{\varepsilon}^{-1}r_0^{\beta}$, where the positive constant $C$ depends on $\textsc{dataset}$. Moreover, it also enables us to decompose
\begin{align}
\left\{\mathbf{M}_{\beta}\big(\omega^{p(\cdot)} |\nabla u|^{p(\cdot)}\big) > \lambda\right\} & = \left\{\mathbf{M}_{\beta}\big(\omega^{p(\cdot)} |\nabla u|^{p(\cdot)}\big) > \lambda; \, \kappa_{\varepsilon}^{-1}\mathbf{M}_{\beta}\mathbb{F}_{\omega} \le  \lambda\right\} \notag \\
& \qquad \qquad \cup \left\{\mathbf{M}_{\beta}\big(\omega^{p(\cdot)} |\nabla u|^{p(\cdot)}\big) > \lambda, \, \kappa_{\varepsilon}^{-1}\mathbf{M}_{\beta}\mathbb{F}_{\omega} > \lambda\right\},\notag
\end{align}
which leads to
\begin{align}
\left|\left\{\mathbf{M}_{\beta}\big(\omega^{p(\cdot)} |\nabla u|^{p(\cdot)}\big) > \lambda\right\}\right| & \le \left|\left\{\mathbf{M}_{\beta}\big(\omega^{p(\cdot)} |\nabla u|^{p(\cdot)}\big) > \lambda, \, \kappa_{\varepsilon}^{-1}\mathbf{M}_{\beta}\mathbb{F}_{\omega} \le \lambda\right\}\right| \notag \\
& \qquad \qquad + \left| \left\{\mathbf{M}_{\beta}\big(\omega^{p(\cdot)} |\nabla u|^{p(\cdot)}\big) > \lambda, \, \kappa_{\varepsilon}^{-1}\mathbf{M}_{\beta}\mathbb{F}_{\omega} > \lambda\right\}\right|. \notag 
\end{align}
Assembling this inequality into~\eqref{Ge1-est1} allows us to conclude
\begin{align}
\left|\left\{\mathbf{M}_{\beta}\big(\omega^{p(\cdot)} |\nabla u|^{p(\cdot)}\big) > \lambda\right\}\right| & \le  C \varepsilon \left|\left\{\sigma^{-1}\mathbf{M}_{\beta}\big(\omega^{p(\cdot)} |\nabla u|^{p(\cdot)}\big) > \lambda\right\}\right| + \left| \left\{\kappa_{\varepsilon}^{-1}\mathbf{M}_{\beta}\mathbb{F}_{\omega} > \lambda\right\}\right|,\label{Ge1-est2}
\end{align}
for every $\lambda \ge \lambda_0$. With given $\gamma>1$, it follows from~\eqref{Ge1-est2} that
\begin{align}
& \left(\int_{\lambda_0}^{\infty}\gamma \lambda^{\gamma-1}\left|\left\{\mathbf{M}_{\beta}\big(\omega^{p(\cdot)} |\nabla u|^{p(\cdot)}\big) > \lambda\right\}\right| d\lambda\right)^{\frac{1}{\gamma}} \notag \\
& \hspace{4cm} \le  C \varepsilon^{\frac{1}{\gamma}} \left(\int_{\lambda_0}^{\infty}\gamma \lambda^{\gamma-1}\left|\left\{\sigma^{-1} \mathbf{M}_{\beta}\big(\omega^{p(\cdot)} |\nabla u|^{p(\cdot)}\big) > \lambda\right\}\right| d\lambda\right)^{\frac{1}{\gamma}} \notag \\
& \hspace{7cm} + C  \left(\int_{\lambda_0}^{\infty}\gamma \lambda^{\gamma-1}\left| \left\{\kappa_{\varepsilon}^{-1}\mathbf{M}_{\beta}\mathbb{F}_{\omega} >  \lambda\right\}\right|d\lambda\right)^{\frac{1}{\gamma}}.\label{Ge1-est3}
\end{align}
On the other hand, for $f \in L^{\gamma}(\Omega)$, it is known that the Lebesgue norm can also be represented as
\begin{align*}
\|f\|_{L^{\gamma}(\Omega)} & = \left(\int_0^{\infty} \gamma \lambda^{\gamma-1} |\{|f|>\lambda\}|  d\lambda \right)^{\frac{1}{\gamma}}.
\end{align*}
Therefore, from~\eqref{Ge1-est3}, one obtains that
\begin{align}
\left(\int_{\lambda_0}^{\infty}\gamma \lambda^{\gamma-1}\left|\left\{\mathbf{M}_{\beta}\big(\omega^{p(\cdot)} |\nabla u|^{p(\cdot)}\big) > \lambda\right\}\right| d\lambda\right)^{\frac{1}{\gamma}} & \le  C \sigma^{-1} \varepsilon^{\frac{1}{\gamma}} \left\|\mathbf{M}_{\beta}\big(\omega^{p(\cdot)} |\nabla u|^{p(\cdot)}\big)\right\|_{L^{\gamma}(\Omega)} \notag \\
& \qquad \qquad + C \kappa_{\varepsilon}^{-1} \left\|\mathbf{M}_{\beta}\mathbb{F}_{\omega}\right\|_{L^{\gamma}(\Omega)}.\label{Ge1-est4}
\end{align}
At this stage, let us apply the following estimate 
\begin{align*}
\|f\|_{L^{\gamma}(\Omega)} &  \le  C \left(\int_{\lambda_0}^{\infty} \gamma \lambda^{\gamma-1} |\{|f|>\lambda\}| d\lambda \right)^{\frac{1}{\gamma}} + C\left(\int_0^{\lambda_0} \gamma \lambda^{\gamma-1} |\{|f|>\lambda\}| d\lambda \right)^{\frac{1}{\gamma}} \\
& \le C \left(\int_{\lambda_0}^{\infty} \gamma \lambda^{\gamma-1} |\{|f|>\lambda\}| d\lambda\right)^{\frac{1}{\gamma}} + C|\Omega|^{\frac{1}{\gamma}} \lambda_0^{\gamma},
\end{align*}
and combined with~\eqref{Ge1-est4}, it yields
\begin{align}
\left\|\mathbf{M}_{\beta}\big(\omega^{p(\cdot)} |\nabla u|^{p(\cdot)}\big)\right\|_{L^{\gamma}(\Omega)} & \le C \sigma^{-1} \varepsilon^{\frac{1}{\gamma}} \left\|\mathbf{M}_{\beta}\big(\omega^{p(\cdot)} |\nabla u|^{p(\cdot)}\big)\right\|_{L^{\gamma}(\Omega)} \notag \\
& \qquad \qquad + C \kappa_{\varepsilon}^{-1} \left\|\mathbf{M}_{\beta}\mathbb{F}_{\omega}\right\|_{L^{\gamma}(\Omega)} + C|\Omega|^{\frac{1}{\gamma}} \lambda_0^{\gamma}.\label{Ge1-est5}
\end{align}
By choosing $\varepsilon$ sufficiently small, the second term on the right-hand side of~\eqref{Ge1-est5} can be absorbed into the left-hand side, and the first term reduces to a constant. We then obtain~\eqref{ineq-Ge1}, which finishes the proof.
\end{proof}

As the reader will see, the proof idea behind Theorem~\ref{theo-Ge1} can be readily extended to a broader class of function spaces. In the next result, we will establish gradient estimates in Orlicz spaces associated with Muckenhoupt weights. Further, it is worth emphasizing that our main results complement and extend the global regularity theory to several classes of function spaces (for example, generalized Lorentz spaces, Orlicz-Lorentz spaces, or even generalized Morrey spaces). For those interested, a more detailed treatment of some simpler problems can be found in our recent works.

We are now in the position to state and proof the second main result of this paper. 

\begin{theorem}\label{theo-Ge2}
Let $\mu \in \mathcal{A}_\infty$ be a given Muckenhoupt weight and a map $\Phi: [0,\infty) \to [0,\infty)$ be an increasing continuously differentiable function with $\Phi(0)=0$ satisfying
\begin{align}\label{Phi-cond}
\Phi(2\lambda) \le c_0 \Phi(\lambda), \quad \lambda\ge 0,
\end{align}
for some constant $c_0>0$. Then, under the structure assumptions of Theorem~\ref{theo-Ge1}, there exist $\kappa \in (0,1)$ and $C = C(c_0,[\mu]_{\mathcal{A}_{\infty}},\textsc{dataset})>0$ such that the following fractional gradient estimate holds true:
\begin{align}\label{ineq-Ge2}
\int_{\Omega}\Phi\left(\mathbf{M}_{\beta}\big(\omega^{p(\cdot)} |\nabla u|^{p(\cdot)}\big)\right)d\mu & \le C \left[ 1 +  \int_{\Omega}\Phi\left(\mathbf{M}_{\beta}\mathbb{F}_{\omega}\right)d\mu\right],
\end{align}
if the assumptions $\mathbf{(H_1)}$-$\mathbf{(H_3)}$ satisfy for some $r_0>0$. 
\end{theorem}
\begin{proof}
In this proof, to simplify the level-set inequality associated with a Muckenhoupt weight, we will use the concept of \emph{weighted distribution function}, first introduced by Grafakos~\cite{G14}, and further developed in our previous work~\cite{PNJFA}. If this may be desirable, the weighted distribution function of any $f \in \mathcal{M}eas(\Omega)$ is defined as follows
\begin{align}\label{df-lambda}
d_f(\lambda) = \mu(\{x \in \Omega: \, |f(x)|>\lambda\}), \quad \mbox{ for any } \lambda \ge 0.
\end{align}
Thanks to~\eqref{df-lambda} and Fubini's theorem, it is easily to verify
\begin{align}\label{int-Phi}
\int_{\Omega} \Phi(|f|) d\mu = \int_0^{\infty} \Phi^{\prime}(\lambda) d_f(\lambda) d\lambda.
\end{align}
We now use the definition of the distribution function in~\eqref{df-lambda} for the following quantities
\begin{align}\notag
\mathbb{U} := \mathbf{M}_{\beta}\big(\omega^{p(\cdot)} |\nabla u|^{p(\cdot)}\big), \quad \text{and} \quad \mathbb{G} := \mathbf{M}_{\beta}\mathbb{F}_{\omega}.
\end{align} 
Taking~\eqref{int-Phi} into account, it allows us to write
\begin{align}\label{Ge2-est0}
\int_{\Omega} \Phi(\mathbb{U}) d\mu = \int_0^{\infty} \Phi^{\prime}(\lambda) d_{\mathbb{U}}(\lambda) d\lambda.
\end{align}
At this point, to control the integral on the right-hand side of~\eqref{Ge2-est0}, it is necessary to establish a relationship between $d_{\mathbb{U}}$ and $d_{\mathbb{G}}$. By applying Theorem~\ref{theo-LV}, for every $\varepsilon \in (0,1)$, there exist positive constants $\sigma$, $\kappa_{\varepsilon}$, and $C$ such that if $\mathbf{(H_1)}$-$\mathbf{(H_3)}$ hold for $\kappa_{\varepsilon}$ and for some $r_0>0$, then the following level-set inequality holds:
\begin{align}\notag
\left|\left\{\mathbb{U} > \lambda, \, \kappa_{\varepsilon}^{-1}\mathbb{G} \le  \lambda\right\}\right| \le C \varepsilon \left|\left\{\sigma^{-1}\mathbb{U} > \lambda\right\}\right|,
\end{align}
for all $\lambda \ge \lambda_0 := \kappa_{\varepsilon}^{-1}r_0^{\beta}$. Since $\mu \in \mathcal{A}_{\infty}$, there exists $\theta>0$ such that
\begin{align}\label{Ge2-est1}
\mu\left(\left\{\mathbb{U} > \lambda, \, \kappa_{\varepsilon}^{-1}\mathbb{G} \le  \lambda\right\}\right) \le C \varepsilon^{\theta} \mu\left(\left\{\sigma^{-1}\mathbb{U} > \lambda\right\}\right).
\end{align}
In a similar fashion to the proof of~\eqref{Ge1-est2} in Theorem~\ref{theo-Ge1}, it follows from~\eqref{Ge2-est1} that
\begin{align}\label{DF-est1}
d_{\mathbb{U}}(\lambda) \le C \varepsilon^{\theta} d_{\sigma^{-1}\mathbb{U}}(\lambda) + d_{\kappa_{\varepsilon}^{-1}\mathbb{G}}(\lambda), \quad \mbox{ for all } \lambda \ge \lambda_0.
\end{align}
By substituting~\eqref{DF-est1} into the integral over $(\lambda_0,\infty)$ in~\eqref{Ge2-est0}, one obtains that
\begin{align}\label{Ge2-est3}
\int_{\Omega} \Phi(\mathbb{U}) d\mu & = \int_0^{\lambda_0} \Phi^{\prime}(\lambda) d_{\mathbb{U}}(\lambda) d\lambda + \int_{\lambda_0}^{\infty} \Phi^{\prime}(\lambda) d_{\mathbb{U}}(\lambda) d\lambda \notag \\
& \le \mu(\Omega) \int_0^{\lambda_0} \Phi^{\prime}(\lambda) d\lambda + \int_{\lambda_0}^{\infty} \Phi^{\prime}(\lambda) C \varepsilon^{\theta} d_{\sigma^{-1}\mathbb{U}}(\lambda) d\lambda + \int_{\lambda_0}^{\infty} \Phi^{\prime}(\lambda) d_{\kappa_{\varepsilon}^{-1}\mathbb{G}}(\lambda) d\lambda \notag \\
& \le \mu(\Omega) \Phi(\lambda_0) + C \varepsilon^{\theta} \int_{\Omega} \Phi(\sigma^{-1}\mathbb{U}) d\mu + \int_{\Omega} \Phi(\kappa_{\varepsilon}^{-1}\mathbb{G}) d\mu.
\end{align}
By the doubling assumption of $\Phi$ in~\eqref{Phi-cond}, it is possible to show that $\Phi(\sigma^{-1}\mathbb{U}) \le C \Phi(\mathbb{U})$. Indeed, one can choose an integer $m$ such that $\sigma^{-1} \le 2^m$. Due to the monotonicity of $\Phi$ together~\eqref{Phi-cond}, we derive that
\begin{align*}
\Phi(\sigma^{-1}\mathbb{U}) \le \Phi(2^m\mathbb{U}) \le c_0^{m} \Phi(\mathbb{U}).
\end{align*}
A completely analogous argument shows that there is a positive constant $C_{\varepsilon}$ which depends on $\varepsilon$ such that $\Phi(\kappa_{\varepsilon}^{-1}\mathbb{G}) \le C_{\varepsilon} \Phi(\mathbb{G})$. Then,~\eqref{Ge2-est3} deduces to
\begin{align}\notag
\int_{\Omega} \Phi(\mathbb{U}) d\mu & \le \mu(\Omega) \Phi(\lambda_0) + C \varepsilon^{\theta} \int_{\Omega} \Phi(\mathbb{U}) d\mu + C_{\varepsilon}\int_{\Omega} \Phi(\mathbb{G}) d\mu.
\end{align}
The proof of~\eqref{ineq-Ge2} is now complete by taking $\varepsilon>0$ small enough.
\end{proof}

As a consequence of Theorem~\ref{theo-Ge2} and the boundedness of the maximal operator in \cite[Theorem 1.2.1]{KK91}, we directly derive the following corollary. 

\begin{corollary}\label{cor-main}
Let $\Phi: [0,\infty) \to [0,\infty)$ be a Young function and $\Phi \in \Delta_2 \cap \nabla_2$. Then, under the structure assumptions of Theorem~\ref{theo-Ge1}, there exist $\kappa \in (0,1)$ and $C = C(\textsc{dataset})>0$ such that the following fractional gradient estimate holds true:
\begin{align}\label{ineq-cor}
\int_{\Omega}\Phi\left(\omega^{p(x)} |\nabla u|^{p(x)}\right)dx & \le C \left[ 1 +  \int_{\Omega}\Phi\left(\mathbb{F}_{\omega}\right)dx\right],
\end{align}
if the assumptions $\mathbf{(H_1)}$-$\mathbf{(H_3)}$ satisfy for some $r_0>0$.
\end{corollary}

\section{Further discussion}
\label{sec:future}

In summary, our work in this paper focuses on the higher integrability of weak solutions to the two-sided obstacle-type elliptic variational inequality~\eqref{OP-var}. As far as we know, the presence of a singular-degenerate matrix weight in the leading coefficients, as well as the variable exponent growths, can lead to rather significant technical difficulties to deal with. Moreover, the two-sided obstacle problem, whose solution is trapped between two prescribed obstacle functions, is itself a type of free boundary problem that arises naturally from many mathematical models of physical phenomena. With a special interest in recent years, it is therefore interesting and significant to reach for regularity estimates for this type of problem. 

Motivated by preceding contributions to the analysis of both linear and nonlinear problems with singular or degenerate weights in~\cite{FKS1982, BDGN2022, BBDL2023, BY24, CMP2018, Phan2020, CMP2019, DP2021}, in this study, we have improved and designed some tricks to find the optimal gradient regularity of weak solutions from Lebesgue to Orlicz spaces, via estimates involving fractional maximal functions. It is also natural to expect that this study has considerable room for further improvement and generalization. For instance, in more involved situations, the singular or degenerate matrix weight can be incorporated into the nonuniformly elliptic structures, problems with non-standard growth, with measure data or lower order terms, nonlocal problems with fractional operators, and also in the theory of variational inequality, free boundary problems whose degeneracies or singularities are governed by matrix-valued weights, etc. As far as we are concerned, such models naturally arise in various branches of mathematical physics, and their analysis is highly nontrivial and deserves further attention in the literature.

\section*{Acknowledgement}
This research is funded by Vietnam National Foundation for Science and Technology Development (NAFOSTED), Grant Number: 101.02-2025.03.

\section*{Conflict of Interest}
The authors declared that they have no conflict of interest.

\section*{Declarations}
Data sharing not applicable to this article as no datasets were generated or analysed during the current study.\\

\end{document}